\documentclass[a4paper,10pt]{article}
\usepackage{geometry}                % See geometry.pdf to learn the layout options. There are lots.
\usepackage[active]{srcltx}
\usepackage{hyperref}
\usepackage[normalem]{ulem}

\usepackage{amsmath}
\usepackage{amssymb}
\usepackage{parskip}
\usepackage{mathtools}
\usepackage{fullpage}
\usepackage{mathrsfs}
\usepackage{amsthm}
\usepackage{physics}
\usepackage{bbm}
\usepackage[utf8]{inputenc}
\usepackage{caption}
\usepackage{subcaption}
\usepackage{graphicx}
\usepackage{attachfile}
\usepackage{navigator}
\usepackage{theoremref}
\usepackage{faktor}
\usepackage{musicography}
\usepackage{accents}
\usepackage{tikz}
\usetikzlibrary{angles,quotes}

\newtheorem{theorem}{Theorem}[section]
\newtheorem{maintheorem}{Theorem}

\newtheorem{definition}[theorem]{Definition}
\newtheorem{proposition}[theorem]{Proposition}
\newtheorem{corollary}[theorem]{Corollary}
\newtheorem{lemma}[theorem]{Lemma}
\newtheorem{remark}[theorem]{Remark}

\newtheorem{assumption}[theorem]{Assumption}

\numberwithin{equation}{section}

\newcommand{\R}{\mathbb{R}}
\newcommand{\Q}{\mathbb{Q}}

\newcommand{\N}{\mathbb{N}}

\newcommand{\e}{\varepsilon}
\newcommand{\p}{\varphi}

\newcommand{\diag}{\mathrm{diag}}

\newcommand{\vol}{\mathrm{vol}}

\newcommand{\cd}{\mathsf{CD}}

\newcommand{\sfd}{\mathsf d}
\newcommand{\di}{{\rm d}}

\newcommand{\mm}{\mathfrak{m}}

\newcommand{\LM}[1]{\hbox{\vrule width.2pt \vbox to#1pt{\vfill \hrule
width#1pt height.2pt}}}
\newcommand{\LL}{{\mathchoice
{\,\LM7\,}{\,\LM7\,}{\,\LM5\,}{\,\LM{3.35}\,}}}

\title{Infinitesimal Minkowskianity for manifolds with continuous Lorentzian metrics}
\author{Vanessa Ryborz \thanks{{  Mathematical Institute, University of Oxford, UK. email: {\tt ryborz@maths.ox.ac.uk}}}}
\date{}

\begin{document}
\maketitle
\begin{abstract}
 We prove that any metric measure spacetime arising from a smooth manifold $M$ endowed with a continuous Lorentzian metric $g$ is infinitesimally Minkowskian, under the assumption that $(M, g)$ is causally simple.
 \let\thefootnote\relax
\end{abstract}
\tableofcontents

{\small \textbf{Acknowledgements.} 
The author gratefully acknowledges Andrea Mondino for his supervision and guidance, and thanks Clemens Sämann for valuable discussions. The author is supported by a postgraduate scholarship from the Mathematical Institute at the University of Oxford. 
}

\section{Introduction}

The geometry of metric (measure) spaces is an active field of study  across analysis, metric and differential geometry, and probability.
Metric measure spaces can be seen as a generalisation of smooth Riemannian manifolds, which turn out to carry enough information  to allow generalisations of several core notions from differential geometry. These include upper and lower sectional curvature bounds, which can be defined on length spaces (see for instance \cite{alexander2024alexandrov}), the celebrated curvature dimension condition $\cd(K, N)$ based on optimal transport introduced by Lott-Villani \cite{lott2009ricci, LottVillaniJFA} and Sturm \cite{sturm2006geometryI, sturm2006geometryII} (a synthetic definition of lower Ricci curvature and upper dimension bounds), and a generalised first order calculus (see \cite{ambrosio2014inventio}). 

Analogous approaches were recently started in Lorentzian signature: After the pioneering work by Kronheimer and Penrose \cite{kronheimer1967structure} and Busemann \cite{busemann1967timelike}, Kunzinger and S\"amann \cite{kunzinger2018lorentzian} proposed the notion of Lorentzian {(pre-)}length spaces, a synthetic approach to Lorentzian geometry. They also studied timelike sectional curvature bounds through timelike triangle comparison. 

A burst of activity in the field of non-smooth Lorentzian geometry followed: 
Cavalletti and Mondino \cite{cavalletti2020optimal} developed optimal transport tools and a theory of synthetic timelike lower Ricci curvature bounds  for Lorentzian length spaces (building on \cite{mccann2018displacement} and \cite{mondino2022optimal} that treated the case of smooth Lorentzian manifolds), together with a synthetic Hawking singularity theorem. Mondino and Suhr proposed a synthetic definition of upper timelike Ricci curvature bounds leading to a synthetic formulation of the Einstein equations in \cite{mondino2022optimal}.
In the flavour of the celebrated singularity theorems on smooth spacetimes, main developments in the non-smooth case include the study of cones and related singularity theorems \cite{alexander2019generalized}, the relationship between spacetime inextendability and synthetic curvature blow-up established in \cite{grant2019inextendibility}, singularity theorems under variable timelike lower Ricci curvature bounds \cite{BM-MEMS}, and an extension of the Penrose singularity theorem to continuous spacetimes satisfying a synthetic null energy condition \cite{CMM-Synthetic-2025}.
Further developments include a Lorentzian splitting theorem  for Lorentzian length spaces with non-negative timelike sectional curvature \cite{beran2023splitting}, a definition of Lorentzian Hausdorff measures and dimension \cite{mccann2022lorentzian}, the study of time functions on Lorentzian length spaces \cite{burtscher2025time}, a definition of timelike lower Ricci curvature bounds on Finsler spacetimes \cite{braun2024optimal}, and a new Lorentzian isoperimetric-type inequality under timelike lower Ricci curvature bounds \cite{CM:LorentzianIsop}.

In \cite{beran2024nonlinear}, the authors introduce the notion of  metric measure spacetimes, a framework largely consistent with the works by Kunzinger-Sämann, Cavalletti-Mondino, and Mondino-Suhr. In this setting, they  prove a d'Alembertian comparison (see also \cite{braun2024exact}) and develop a Lorentzian differential first order calculus.   
Based on that, they define a synthetic criterion, called \emph{infinitesimal Minkowskianity}, to distinguish those metric measure spacetimes whose time separation arises from a Lorentzian inner product from Lorentz-Finsler structures.  This criterion  can be seen as a Lorentzian counterpart of infinitesimal Hilbertianity, which was introduced in \cite{AGS14, gigli2015differential} for metric measure spaces and which provided a natural framework  for remarkable structural results \cite{MN-JEMS2019, BS-CPAM2020} when combined with Lott-Sturm-Villani's $\cd(K,N)$ condition.

A vast class of metric measure spacetimes arises from manifolds $M$ endowed with non-smooth Lorentzian metrics $g$.
They have been studied in \cite{chrusciel2012lorentzian, samann2016global, graf2020singularity, calisti2025hawking}. Such manifolds serve to describe various phenomena from general relativity, examples being the inside or outside of a star, spacetimes with conical singularities, and cosmic strings (\cite{vickers1990quasi, vickers2000generalized}). 
In this paper, we study the case that $g$ is continuous. A physically relevant example is provided by Penrose's impulsive gravitational waves (see for instance \cite{penrose1972geometry} and \cite[Chapter 20]{griffiths2009exact}).

Both metric measure spacetimes and manifolds endowed with continuous Lorentzian metrics can be seen as generalisations of smooth Lorentzian manifolds, hence a naturally emerging question is whether these two generalisations are compatible. In Riemannian signature, such compatibility questions were studied in \cite{burtscher2012length, kunzinger2022synthetic, mondino2025equivalence, kunzinger2025ricci, eros2025distributional}. 
The goal of this paper is to compare the classical first order calculus on manifolds with the first order calculus on metric measure spacetimes developed in \cite{beran2024nonlinear}. This work can be seen as a Lorentzian counterpart to \cite{burtscher2012length} and \cite[Section 4]{mondino2025equivalence}.

The main result of this paper is:

\begin{maintheorem}\label{mainthm}
    Let $M$ be a $(d+1)$-dimensional manifold and $g$ be a continuous Lorentzian metric on $M$ such that $(M, g)$ is a causally simple spacetime.
    Then the metric measure spacetime $(M, \ell_g, \vol_g)$ is infinitesimally Minkowskian. 
\end{maintheorem}

Here, the causal simplicity is an assumption of technical nature as we merely use strong causality in our proofs, see Remark \ref{remark_causal_simplicity}.

The proof of Theorem \ref{mainthm} essentially consists of two steps: 
\begin{itemize}
    \item We verify that for any sufficiently regular causal curve $\gamma\colon [0,1] \to M$, and almost every $t \in [0,1]$, the causal speed $|\dot{\gamma}_t|$ which is almost everywhere given by $|\dot{\gamma}_t|=\lim_{s \searrow t} \frac{\ell(\gamma_t, \gamma_s)}{|s-t|}$ essentially coincides with the Lorentzian length of the derivative $|\dot{\gamma}_t|_g =\sqrt{-g(\dot{\gamma}_t,\dot{\gamma}_t)}$. This step is contained in Corollary \ref{causal_speed_computed}.
    \item We prove that, for a causal function $f$, the maximal weak subslope is given by $|\nabla_g f|_g$, wherever $f$ is finite and differentiable, and $+\infty$ wherever $f$ is infinite or not differentiable. The idea of the proof is similar to the proof in the Riemannian signature \cite[Proposition 4.24]{mondino2025equivalence} but in this setting an extra difficulty arises as we may work with functions that are not $C^1$. We use that causal functions are of bounded variation (see \cite{beran2024nonlinear}), so the distributional derivative of $f$ is given by a measure $\mathcal{D}f$, which potentially has a singular part $\mathcal{D}f^s$. Using the Besicovitch covering theorem allows us to suitably localise away from the singular part $\mathcal{D}f^s$ (Lemma \ref{singular_measure_lebesgue_point}) and then perform a contradiction argument in Lemma \ref{weak_subslope_identified}. 
\end{itemize}

The paper is organised as follows: In Section \ref{sec2}, we recall the synthetic Lorentzian theory from \cite{beran2024nonlinear}. Thereafter, in Section \ref{sec3}, we establish relevant measure theoretic results that will be needed in the proof of Theorem \ref{mainthm}. Finally, in Section \ref{sec4}, we recall how a manifold with a continuous Lorentzian metric enters the class of metric measure spacetimes and then compare the synthetic first order calculus to the classical manifold calculus to prove Theorem \ref{mainthm}.

\subsection*{Notation}
Given $m \in \N$ and real numbers $\lambda_1, \ldots, \lambda_m$, the expression $\mathrm{diag}(\lambda_1, \ldots, \lambda_m)$ refers to the $m$-dimensional diagonal matrix with entries $\lambda_1, \ldots, \lambda_m$. Given a matrix $A \in \R^{m \times m}$, $\norm{A}_{op}$ denotes the operator norm of $A$ induced by the Euclidean metric on $\R^m$.

Given a manifold $M$, a continuous and time orientable Lorentzian metric $g$ on $M$, and two vectors $v, w \in T_pM$ for some $p \in M$, we write $g(v,w) = \langle v, w \rangle_g = \sum_{i,j} g_{ij}v_iw_j$ for the $g$-inner product of $v$ and $w$.
Given a causal function $f$ on $M$, we denote by $|\di f|$ the maximal weak subslope.
We will denote by $\di f \in T^*M$ its differential, $\nabla_g f \in TM$ its gradient and by $|\nabla_g f|_g(=|\di f|_g) = \sqrt{-\sum_{i,j}g^{ij}\partial_if\partial_jf}$ the $g$-norm of its gradient (= the $g$-norm of its differential), wherever they exist. For a causal path $\gamma\colon [0,1] \to M$, we denote by $\dot{\gamma}_t \in T_{\gamma_t}M$ the derivative, by $|\dot{\gamma}_t|_g = \sqrt{-g(\dot{\gamma}_t, \dot{\gamma}_t)}$ the $g$-norm of the derivative, and by $|\dot{\gamma}_t|$ the causal speed of $\gamma$ at time $t$, wherever these exist. \\
Given an open set $U \subset M$ and Lorentzian metrics $g, g'$, we say that $g \prec g'$ (or $g' \succ g$) on $U$ if for every $0 \neq v \in TU$ it holds that $g(v,v) \leq 0 \implies g'(v,v) < 0$.
\section{Synthetic Lorentzian spaces and causal functions}\label{sec2}
We will briefly summarise relevant aspects of the theory on metric (measure) spacetimes following \cite{beran2024nonlinear}. We will adapt the infinity conventions from \cite{beran2024nonlinear} which are given by
\begin{align*}
    +\infty -(+\infty) := +\infty,&\\
    -\infty -(-\infty) := +\infty,&\\
    \pm \infty +z := \pm \infty& \quad \mathrm{for}\ z \in [-\infty, +\infty],\\
    (\pm \infty) \cdot 0 := 0.
\end{align*}
Given a set $M$ and a function $\ell\colon  M \times M \to \{-\infty\} \cup [0, \infty]$, we call $\ell$ a \emph{time separation} if for all $x, y, z \in M$ it holds
\begin{align*}
   \ell(x,x) =0, \quad \min(\ell(x,y)+\ell(y,z), \ell(y,z)+ \ell(x,y)) \leq \ell(x,z).
\end{align*}
Define the transitive relations chronology $\ll$ and causality $\leq$ via
\begin{align*}
    \ll := \ell^{-1}((0, +\infty]), \ \leq := \ell^{-1}([0, +\infty]).
\end{align*}
We call a subset $A$ of $M$ \emph{achronal}, if no two points in $A$ are related by $\ll$. 
\begin{definition}
    We call $(M, \ell)$ a \emph{metric spacetime} if $\ell$ is a time separation and $\leq$ is a partial order. We call $(M, \tau, \ell)$ a \emph{Polish metric spacetime} if $\tau$ is a Polish topology on the metric spacetime $(M, \ell)$ and if $\{\ell \ge 0\}$ is a closed set in $M^2$. 
\end{definition}
\smallskip

\begin{definition}[\cite{beran2024nonlinear}, Definition 3.1]
    A function $f\colon M \to \overline{\R}$ is called a \emph{causal function} if $f$ is causally order-preserving, i.e., for all points $x, y \in M$ with $x \leq y$ it holds $f(x) \leq f(y)$. We denote by $\mathrm{Dom}(f)$ the domain of $f$, i.e., the set of all $x \in M$ such that $f(x)$ is a real number.
\end{definition}
\smallskip

\begin{definition}[\cite{beran2024nonlinear}, Definition 2.26]
    A map $\gamma\colon [0,1] \to M$ is called a \emph{left continuous causal path} if it is causal, i.e., for $s \leq t$ it holds $\gamma_s \leq \gamma_t$, and left continuous, i.e., for all $t \in (0,1]$ the limit $\gamma_{t-} := \lim_{s \nearrow t}\gamma_s$ exists and equals $\gamma_t$. The space of all left continuous causal paths is denoted $\mathrm{LCC}([0,1], M)$.
\end{definition}

\smallskip
\begin{lemma}[\cite{beran2024nonlinear}, Lemma 3.2]\label{cts_outside_achronal}
    Every causal function $f$ has a representative that is continuous outside a countable union of achronal sets. 
\end{lemma}

\smallskip
\begin{definition}[\cite{beran2024nonlinear}, Theorem 2.23 - Definition 2.25]\label{def_causal_speed}
    Let $\gamma\colon  [0,1] \to M$ be a causal path. Then there exists a maximal Borel measure $|\dot{\boldsymbol{\gamma}}|$ on $[0,1]$ such that for all $0 \leq s \leq t \leq 1$, it holds $|\dot{\boldsymbol{\gamma}}|([s,t]) \leq \ell(\gamma_s, \gamma_t)$. Denote its Lebesgue decomposition with respect to $\mathcal{L}^1$ as $|\dot{\boldsymbol{\gamma}}| = |\dot{\gamma}|\mathcal{L}^1 + |\dot{\boldsymbol{\gamma}}|^s$. 
    Then, the function $|\dot{\gamma}|$ is called the \emph{causal speed} of $\gamma$. It is for almost every $t \in [0,1]$ given by 
    \begin{align}
        |\dot{\gamma}_t| = \lim_{s \nearrow t} \frac{\ell(\gamma_s, \gamma_t)}{t-s} = \lim_{s \searrow t} \frac{\ell(\gamma_t, \gamma_s)}{s-t}.
    \end{align}
\end{definition}
\smallskip

\begin{definition}[\cite{beran2024nonlinear}, Definition 2.48]
   A \emph{metric measure spacetime} is a quadruple $(M, \tau, \ell, \mathfrak{m})$, where $(M, \tau, \ell)$ is a Polish metric spacetime and $\mathfrak{m}$ a non-negative and non-zero Radon measure on $M$.
\end{definition}
Given a topological space $X$, we denote by $\mathcal{P}(X)$ the space of Borel probability measures on $X$. Given two topological spaces $X, Y$, a Borel map $T: X \to Y$, and a Borel measure $\mu$ on $X$, the pushforward $T_{\#}\mu$ of $\mu$ by $T$ is defined via $T_{\#}\mu(B) = \mu (T^{-1}(B))$, for every Borel set $B \subset Y$.
For $t \in [0,1]$, the evaluation map $e_t\colon \mathrm{LCC}([0,1], M) \to M$ is defined via $e_t(\gamma) := \gamma_t$.
\smallskip

\begin{definition}[\cite{beran2024nonlinear}, Definition 3.7]
    A measure $\boldsymbol{\pi} \in \mathcal{P}(\mathrm{LCC}([0,1], M))$ is called a \emph{test plan} if it has bounded compression, i.e., if there exists a constant $C>0$ such that for all $t \in [0,1]$ it holds 
    \begin{align*}
        (e_t)_\#\boldsymbol{\pi} \leq C \mathfrak{m}. 
    \end{align*}
\end{definition}

\smallskip
\begin{definition}[\cite{beran2024nonlinear}, Definition 3.11]\label{def:subslope}
    Let $(M, \tau, \ell, \mathfrak{m})$ be a metric measure spacetime and $f\colon M \to \overline{\R}$ a causal function. A Borel function $G\colon M \to [0, \infty]$ is called a \emph{weak subslope} of $f$ if for every test plan $\boldsymbol{\pi} \in \mathcal{P}(\mathrm{LCC}([0,1], M))$, it holds
    \begin{align}
        \int [f(\gamma_1)-f(\gamma_0)]\, \di \boldsymbol{\pi}(\gamma) \geq \int \int_0^1 G(\gamma_t)|\dot{\gamma}_t|\, \di t \di \boldsymbol{\pi}(\gamma).
    \end{align}
    By Theorem 3.14 in \cite{beran2024nonlinear}, there exists a unique weak subslope $G$ that is no smaller than any other weak subslope. We will call it the \emph{maximal weak subslope} and denote it by $|\di f|$.
\end{definition}

\smallskip
\begin{definition}[\cite{beran2024nonlinear}, Defintion 1.4]
    We call $(M,\tau, \ell, \mathfrak{m})$ \emph{infinitesimally Minkowskian} if for any two causal functions $f, g\colon  M \to \overline{\R}$, it holds
    \begin{align}\label{def:infinitesimal_minkowskianity}
        2|\di(f+g)|^2+ 2|\di f|^2 = |\di g|^2 + |\di (2f+g)|^2 \ \mathfrak{m}-\mathrm{a.e.}
    \end{align}
\end{definition}
Given a test plan $\boldsymbol{\pi}$, we denote by $\rho_{\boldsymbol{\pi}}$ the density of the push-forward of the measure $|\dot{\gamma}_t|\di\boldsymbol{\pi}(\gamma) \di t$ under the ``full" evaluation map $\mathsf{e}(\gamma, t) := \gamma_t$. Define $\mathrm{Vis}_{\boldsymbol{\pi}}:=\{\rho_{\boldsymbol{\pi}}>0\}$. We define $\mathrm{Vis}(M)$ as the essential union of the sets $\mathrm{Vis}_{\boldsymbol{\pi}}$, as $\boldsymbol{\pi}$ varies over all test plans (see \cite[Proposition 3.1.9]{gigli2020lectures}). 
We will use that a property holds $\mathfrak{m}$-almost everywhere on $\mathrm{Vis}(M)$ if and only if it holds $\mathfrak{m}$-almost everywhere on $\mathrm{Vis}_{\boldsymbol{\pi}}$ for every test plan $\boldsymbol{\pi}$. 
Finally denote $\mathrm{Invis}(M):= M\setminus \mathrm{Vis}(M)$.

\smallskip
\begin{lemma}[\cite{beran2024nonlinear}, Lemma 3.19]\label{achronal_invisible}
    Let $(M, \tau, \ell, \mm)$ be a metric measure spacetime and $A \subset M$ be an achronal set. Then $\mm(A \setminus\mathrm{Invis}(M) ) =0$.
\end{lemma}

\smallskip
\begin{lemma}\label{local_finiteness}
    Let $(M, \tau, \ell, \mm)$ be a metric measure spacetime and $f$ be a causal function. Then for $\mm$-almost every $p\in \{|f|< \infty\}\cap \mathrm{Vis}(M)$, there exists an open neighbourhood $U$ of $p$ such that $\mm(U \setminus \{|f|< \infty\})=0$.
\end{lemma}
\begin{proof}
By Lemma \ref{cts_outside_achronal}, we can find a representative $\Tilde{f}$ of $f$ and achronal sets $\{A_n\}_{n\in \N}$ such that $\Tilde{f}$ is continuous outside $\bigcup_{n \in \N} A_n$. We will from now on write $f$ for $\Tilde{f}$. Now Lemma \ref{achronal_invisible} yields that $\mm$-almost every point in $\{|f|< \infty\}\cap \mathrm{Vis}(M)$ is in $F:=\{|f|< \infty\}\cap \mathrm{Vis}(M) \setminus \bigcup_{n \in \N} A_n$. Pick a point $p \in F$. As $f$ is continuous and finite at $p$, we get that there exists an open neighbourhood $U$ of $p$ such that $f$ is finite on $U$.
\end{proof}

\section{Preliminary results from measure theory}\label{sec3}
The following measure theoretic results will be useful in Section \ref{sec4}. Let $m \in \N$. 
\smallskip

\begin{lemma}\label{ball_evaluation_measurable}
    Let $\Omega\subset \R^m$ be open and let $\mu$ be a locally finite Radon measure on $\Omega$. 
    For $\delta>0$ define the set $\Omega_\delta:= \{p \in \Omega: \overline{B_\delta(p)}\subset \Omega\}$. 
    Then for every $\delta>0$, the map $E_{\mu, \delta}\colon \Omega_\delta \to [0, \infty),\,   p\mapsto \mu(\overline{B_\delta(p)})$ is $\mathcal{L}^m$-measurable and locally bounded.
\end{lemma}
\begin{proof}
    Fix a $\delta>0$. 
    Fix a compact set $K_1 \subset \Omega_\delta$ and a compact set $K_2 \subset \Omega$ such that $\overline{B_\delta(K_1)} \subset K_2$. Now for each $p \in K_1$, it holds that 
    \begin{align*}
        \mu(\overline{B_\delta(p)}) \leq \mu(K_2) < \infty.
    \end{align*}
    This proves the local boundedness. 

    For $i=0, \ldots, m-1$, define the sets $P_{i}$ and the Radon measures $\mu_{i+1}, \nu_{i}$ such that 
    \begin{itemize}
        \item $P_{i}$ is a countable union of $i$-dimensional spheres, where $i$-dimensional spheres are given by $\{q \in \R^m: |q-p|=r, q-p \in V\}$, where $p \in \R^m$, $r>0$ and $V \subset \R^m$ is an $i$-dimensional linear subspace of $\R^m$. If $i=0$, $P_0$ is going to be a countable set of points.
        \item $\nu_{i}= \mu_{i}\LL P_{i}$.
        \item $\mu_{i+1} = \mu_{i}-\nu_{i}$ is zero on spheres of dimension at most $i$.
    \end{itemize}

    Define $\mu_0:= \mu$. 
    Now, define the set 
    \begin{align*}
        P_0:= \{p \in K_2: \mu(\{p\})>0\}. 
    \end{align*}
    This set has to be countable, as $P_0 = \bigcup_{n \in \N} \{p \in K_2: \mu(\{p\})>\frac{1}{n}\}$, and for each $n\in \N$, we have that $\#\{p \in K_2: \mu(\{p\})>\frac{1}{n}\}\leq n\mu(K_2) <\infty$. 
    Now define the measures
    \begin{align}
        \nu_0:= \mu \LL P_0,\ \mu_1:= \mu-\nu_0.  
    \end{align}
    Now for $i \in \{1, \ldots, m-1\}$, define inductively the measures $\nu_i, \mu_{i+1}$ as follows:
    Let
    \begin{align}
        P_i:=\{S:  i\mathrm{-dimensional\ sphere}, \ S \subset K_2,\ \mu_i(S)>0\}. 
    \end{align}
    For every $n \in \N$, we have that 
    \begin{align}\label{finitely_many_spheres_of_big_measure}
       \#\{S:  i\mathrm{-dimensional\ sphere}, \ S \subset K_2,\ \mu_i(S)>1/n\} < n\mu_{i}(K_2) \leq n \mu(K_2) < \infty.
    \end{align}
    Indeed, suppose that this was not true. Then there exist two distinct $i$-dimensional spheres $S_\alpha, S_\beta \in P_i$ such that $\mu_{i}( S_\alpha \cap S_\beta) >0$. 
    Now, we have that for any two distinct $i$-dimensional spheres, their intersection is the finite union of spheres of dimension at most $i-1$. By our assumption, the measure $\mu_{i}$ is zero on spheres of dimension at most $i-1$, which yields a contradiction and hence verifies \eqref{finitely_many_spheres_of_big_measure}.
    This proves that $P_i = \bigcup_{n \in \N} \{S:  i\mathrm{-dimensional\ sphere}, \ S \subset K_2,\ \mu_i(S)>1/n\}$ is countable. 
    Now we define 
    \begin{align}
        \nu_i:= \mu_i \LL P_i,\ \mu_{i+1}:= \mu_i-\nu_{i}. 
    \end{align}
    In total, this gives us a decomposition of the measure 
    \begin{align}
        \mu = \mu_m + \sum_{i=0}^{m-1} \nu_{i},
    \end{align}
    where the $\nu_i$ are supported on countably many $i$-dimensional spheres (collected in $P_i$) and $\mu_m$ is zero on spheres of dimension smaller than $m$. 

    For each $i\leq m-1$, and any $i$-dimensional sphere $S$, we have that the set
    \begin{align*}
        D_{S, \delta}:= \{p: S \subset \partial B_\delta(p)\}
    \end{align*}
    is a sphere of dimension $m-i-1$. Any such sphere is a Borel $\mathcal{L}^m$-zero set.
    Then, the set 
    \begin{align}
        D:= \bigcup_{i=0}^{m-1} \bigcup_{S \in P_i} D_{S, \delta}
    \end{align}
    is a Borel $\mathcal{L}^m$-zero set, which is a countable union of spheres of dimension at most $m-1$.
    For any point $p \in K_1\setminus D$, we have that 
    \begin{align*}
        \mu(\partial B_\delta(p)) = \mu_m(\partial B_\delta(p))=0.
    \end{align*}
    Note that 
    \begin{align}
        \lim_{n \to \infty} \mu(B_{\delta+ 2^{-n}}(p)\setminus B_{\delta- 2^{-n}}(p)) = \mu(\partial B_\delta(p))=0.
    \end{align}
    Fix an $\e>0$ and $n\in \N$ such that $\mu(B_{\delta+ 2^{-n}}(p)\setminus B_{\delta- 2^{-n}}(p)) \leq \e$.
    Then we have that for each $q \in B_{2^{-n-1}}(p)$, it holds that 
    \begin{align}
        |\mu(\overline{B_\delta(p)}) - \mu(\overline{B_\delta(q)})| &\leq \mu((\overline{B_\delta(p)}\setminus \overline{B_\delta(q)})\cup (\overline{B_\delta(q)}\setminus \overline{B_\delta(p)})) \leq \mu(B_{\delta+ 2^{-n}}(p) \setminus B_{\delta- 2^{-n}}(p)) \leq \e.
    \end{align}
    Hence, $E_{\mu, \delta}$ is continuous at $p$. This holds for $p \in K_1 \setminus D$, so the map is continuous $\mathcal{L}^m$-almost everywhere. Hence, $E_{\mu, \delta}$ is $\mathcal{L}^m$-measurable when restricted to $K_1$. As $K_1 \subset \Omega_\delta$ was arbitrary, this concludes the proof.
\end{proof}

\begin{lemma}\label{singular_measure_lebesgue_point}
    Let $\mu$ be a locally finite Radon measure on $\R^m$ such that $\mu \perp \mathcal{L}^m$. 
    %Fix a vector $(a_1, \ldots, a_m) \in (0, 1]^m$.
    Then for $\mathcal{L}^m$-almost every point $p\in \R^m$, it holds that 
   \begin{align}
        \lim_{\delta\to 0} \frac{\mu(B_\delta(p))}{\delta^m } = 0.
    \end{align}
\end{lemma}
\begin{proof}
 Suppose this was not the case. Then there exist an $\e>0$ and a bounded set $A' \subset \R^m$ with $\mathcal{L}^m(A')>0$ such that for each $p \in A'$ it holds 
    \begin{align}\label{contradiction_assumption_singular_measure}
        \limsup_{\delta\to 0} \frac{\mu(B_\delta(p))}{\delta^m } \geq  2\e>0.
    \end{align}
    Note that by Lemma \ref{ball_evaluation_measurable}, we have that the set of all points satisfying \eqref{contradiction_assumption_singular_measure} is Lebesgue measurable. 
    As $\mu \perp \mathcal{L}^m$, there exists a decomposition of $\R^m$ into disjoint Borel sets $L, S$ such that $\mathcal{L}^m(S) = 0 = \mu(L)$ and $L \cup S = \R^m$.
    Define $A:= A' \setminus S$. Then $\mathcal{L}^m(A')= \mathcal{L}^m(A) >0$. Pick an open set $\Tilde{O} \supset A$. Then for each point $p \in A$, there exists a $\delta_p> 0$ such that $B_{\delta_p}(p) \subset \Tilde{O}$ and such that 
    \begin{align}\label{choice_delta_p}
        \frac{\mu(B_{\delta_p}(p))}{(\delta_p)^m} \geq  \e.
    \end{align}
    Then 
    \begin{align*}
        A \subset \bigcup_{p \in A} B_{\delta_p}(p) \subset \Tilde{O}.
    \end{align*} 
    By the Besicovitch covering theorem, there exists a constant $c_m \in \N$ such that there exist $c_m$ countable families $T_1, \ldots, T_{c_m}$ of points $p \in A$ such that 
    \begin{align*}
         A \subset \bigcup_{i=1}^{c_m} \bigcup_{p\in T_i} B_{\delta_{p}}(p) =:O \subset \Tilde{O},  
    \end{align*}
    and such that for each $i$ and $p, p' \in T_i$, it holds 
    \begin{align}\label{pairwise_disjoint_balls}
        B_{\delta_{p'}}(p') \cap B_{\delta_{p}}(p) = \emptyset.
    \end{align}
    By the pigeonhole principle, we get that there exists an $i \in \{1, \ldots, c_m\}$ such that 
    \begin{align}\label{family_with_big_measure}
        \mathcal{L}^m\Big( \bigcup_{p\in T_i} B_{\delta_{p}}(p) \Big) \geq \frac{1}{c_m} \mathcal{L}^m(A). 
    \end{align}
    But now \eqref{pairwise_disjoint_balls} together with \eqref{choice_delta_p} and \eqref{family_with_big_measure} yields that 
    \begin{align}
        \mu\Big( \bigcup_{p\in T_i} B_{\delta_{p}}(p) \Big) \geq  \frac{\e}{\mathcal{L}^m(B_1(0))} \sum_{p \in T_i} \mathcal{L}^m( B_{\delta_{p}}(p)) \geq \frac{\e}{c_m\mathcal{L}^m(B_1(0))} \mathcal{L}^m(A). 
    \end{align}
    It follows that 
    \begin{align*}
        \mu(\Tilde{O}) \geq \mu(O) \geq \frac{\e}{c_m\mathcal{L}^m(B_1(0))} \mathcal{L}^m(A).
    \end{align*}
    As $\Tilde{O}$ is an arbitrary open set containing $A$, the outer regularity of Radon measures on $\R^m$ yields that 
    \begin{align}
        \mu(A) \geq \frac{\e}{c_m\mathcal{L}^m(B_1(0))} \mathcal{L}^m(A) >0.
    \end{align}
    But $A \subset L$ and $\mu(L)=0$ by assumption. This is a contradiction and finishes the proof. 
\end{proof}
For a set $A \subset \R^m$, we denote by $A^c$ its complement, i.e., $A^c = \R^m \setminus A$.
\smallskip

\begin{lemma}\label{null_set_local_to_global}
    Let $A \subset \R^m$ be Lebesgue measurable and $P, Q \subset A$ such that $\mathcal{L}^m(A \cap P^c) = 0$. Assume that for each $p \in P$ there exists a radius $r_p>0$ such that $\mathcal{L}^m(Q^c \cap B_{r_p}(p))=0$. Then, $\mathcal{L}^m(A \cap Q^c)=0$. In particular, $Q$ is measurable.  
\end{lemma}
\begin{proof}
    We may assume $A$ to be bounded.
    Then $A \cap P \subset \bigcup_{p \in P}B_{r_p}(p)$. The Besicovitch covering theorem yields that there exists a countable set $(p_i)_{i \in \N} \subset P$ such that 
    \begin{align*}
        A \cap P \subset \bigcup_{i \in \N}B_{r_{p_i}}(p_i).
    \end{align*}
    Denote by $\mathcal{L}^{m*}$ the outer $m$-dimensional Lebesgue measure. 
    Now we have that 
    \begin{align*}
        \mathcal{L}^{m*}(A \cap Q^c) \leq \mathcal{L}^{m*}(A \cap P^c) + \sum_{i \in \N} \mathcal{L}^{m*}(B_{r_{p_i}}(p_i)\cap Q^c) =0.
    \end{align*}
    Now null sets are Lebesgue measurable, which proves the lemma. 
\end{proof}

 The final lemma of this section is proved analogously to Lemma \ref{Lebesguepointsadvanced_lor}, so the proof is omitted. 

 \smallskip
 \begin{lemma}\label{lebesgue_point_line}
     Let $\Omega\subset \R^m$ be open, $f: \Omega \to \R$ be Lebesgue measurable and locally bounded, and $\Gamma \subset B_1(0)$ be a countable set. Then for $\mathcal{L}^m$-almost every point $p \in \Omega$ and all $\dot{\gamma} \in \Gamma$, we have that $t=0$ is a Lebesgue point of the map 
     \begin{align}
         L^f_{p, \dot{\gamma}}: (-\e,\e) \to \R, t \mapsto f(p+ t\dot{\gamma}), 
     \end{align}
     where $\e$ is such that $B_\e(p) \subset \Omega$.
 \end{lemma}

\section{Application to manifolds with non-smooth Lorentzian metrics}\label{sec4}
We want to study how the theory on causal functions as it has been developed in \cite{beran2024nonlinear} applies to the setting of a Lorentzian manifold with a metric of low regularity.
From now on the setting will be a smooth $(d+1)$-manifold $M$, for $d \geq 1$, equipped with a continuous Lorentzian metric $g$, i.e., a continuous symmetric bilinear $(0,2)$-tensor of signature $(-, +, \ldots, +)$.
We say that a vector $v \in TM$ is \emph{timelike} if $g(v,v) <0$, \emph{null} or \emph{lightlike} if $g(v,v)=0$ and $v \neq 0$, \emph{spacelike} if $v =0$ or $g(v,v)>0$, and \emph{causal} if it is timelike or null.
We assume that $g$ is \emph{time orientable}, i.e., there exists a continuous timelike vector field $X$ which tells us at each point, which cone is the future cone and which cone is the past cone. More precisely, we say that a causal vector $v \in TM$ is future directed, if $g(v, X)<0$, otherwise we say that it is past directed. Time orientable Lorentzian manifolds are called \emph{spacetimes}.

For causal vectors $v \in TM$, we define $|v|_g:= \sqrt{-g(v,v)}$. Recall that for $p \in M$ and $v, w \in T_pM$ causal, the reverse Cauchy-Schwartz inequality holds, i.e.,
\begin{align}\label{reverse_cauchy_schwartz}
    |g(v, w)| \geq |v|_g|w|_g,
\end{align}
with equality if and only if $v = \alpha w$ for some $\alpha \in \R$. 

A spacetime arising from a manifold with a continuous Lorentzian metric is naturally associated with a topology, a time separation, and a measure as follows:
\begin{itemize}
    \item Any complete Riemannian metric $h$ on $M$ induces a distance $\sfd_h$ on $M$ such that $(M, \sfd_h)$ is a Polish metric space. Fix some such $h$ and denote by $\tau_h$ the topology induced by $\sfd_h$. Note that $\tau_h$ equals the topology induced by charts, so from now on we will not always mention the topology explicitly.  
    \item The metric $g$ induces a volume measure locally given by $\di \vol_g = \sqrt{-\det g}\, \di \mathcal{L}^{d+1}$.
    \item We say that an absolutely continuous curve $\gamma\colon[0,1] \to M$ is \emph{classically causal} (c.c.) and future directed (past directed) if $g(\dot{\gamma}_t,\dot{\gamma}_t) \leq 0$ and $\dot{\gamma}_t$ is future directed (past directed) or zero for almost every $t \in [0,1]$. 
    \item The time separation $\ell_g:M^2\to \{-\infty\}\cup [0, \infty]$ is defined as
    \begin{align}
        \ell_g(x, y) = \sup  \Big\{ L_g(\gamma):=\int_0^1\sqrt{-g(\dot{\gamma}_t,\dot{\gamma}_t)}\, \di t\Big|\ \gamma \in C^{0,1}([0,1], M)\ \mathrm{c.c., \ future \ dir.},\ \gamma_0=x, \gamma_1=y.   \Big\}. 
    \end{align}
\end{itemize}
We then define the chronological future and past $I^+, I^-$ of a point $p \in M$ via
\begin{align*}
    I^+(p) = \{q \in M: p \ll q\}, \ I^-(p) = \{q \in M: q \ll p\},
\end{align*} 
and the causal future and past $J^+, J^-$ of a point $p \in M$ via
\begin{align*}
    J^+(p) = \{q \in M: p \leq q\}, \ J^-(p) = \{q \in M: q \leq p\},
\end{align*} 
where the chronological relation $\ll:=\ll_g$ is defined via $\ll_g := \ell_g^{-1}((0, \infty]) \subset M^2$ and the causal relation $\leq := \leq_g$ is defined via $\leq_g := \ell_g^{-1}([0, \infty]) \subset M^2$. From now on when speaking of classically causal curves, we always refer to future directed classically causal curves, even if not explicitly mentioned.
\smallskip

\begin{remark}\label{causal_curves_constant_remark}
    Notice that in the literature, it is sometimes employed to consider only those classically causal curves whose tangent vectors are non-zero almost everywhere (e.g., \cite{samann2016global}). In the context of Lorentzian length spaces, causal curves are commonly defined to be non-constant \cite{kunzinger2018lorentzian, cavalletti2020optimal}. However, our definition allows classically causal curves with constant segments, as this ensures that $\ell_g(x, x)=0$ for $x \in M$,  and better suits the definition of causal paths in \cite{beran2024nonlinear}. In particular, Lemma \ref{classically_causal_curve_non_constant} shows that our choice to allow classically causal curves with zero derivatives at a positive measure set of times for the definition of $\ell_g$ does not change the values of the time separation $\ell_g(x, z)$ for $x \neq z$.
\end{remark}
\smallskip

\begin{lemma}\label{coordinates_causal_curves_monotone}
    Let $g$ be a continuous Lorentzian metric on $M$ and let $p \in M$. Then there exists an open neighbourhood $U$ of $p$ and coordinates $y= (y_0, \ldots, y_d)$ for $U$ such that the following holds: For any points $p^1 \leq_g p^2 \in U$ that can be joined by a $g$-classically causal curve that is contained in $U$, it holds that $y_i(p^1) \leq y_i(p^2)$ for all $i =0, \ldots, d$.
\end{lemma}
\begin{proof}
    This proof uses a similar idea as in \cite[Theorem A.1]{beran2024nonlinear}. Fix a neighbourhood $W$ around $p$ that is contained in one coordinate patch and assume that $W \subset \R^{d+1}$. Moreover, we may assume the coordinates on $W$ to be chosen in a way that $g(p) = \diag(-1, 1, \ldots, 1)$. 
    We call the coordinates $(t, x_1, \ldots, x_d)$ and suppose that $\partial_t$ is future directed. 
    Now there exists a neighbourhood $U \subset W \subset \R^{d+1}$ of $p$ such that $g \prec \mathrm{diag}(-2, 1, \ldots, 1)=: m_2$ on $U$. Possibly shrinking $U$ further, we may assume that $U \ni p$ is $m_2$-globally hyperbolic and that it can be written as $I_{m_2}^+(p^-) \cap I_{m_2}^-(p^+)$, for some $p^-, p^+ \in W$. 
    Now there exists a linear transform into coordinates $(y_0, y_1, \ldots, y_d)$ such that the future cone of $m_2$ is contained in the set 
    \begin{align*}
        C:=\big\{\sum_{i=0}^d \alpha_i \partial_{y_i}| \alpha_0, \ldots, \alpha_d \geq 0\big\}.
    \end{align*}
    As $m_2 \succ g$ in $U$, we get that the future cone of $g$ is contained in $C$ at every point $q \in U$. 
    Now pick two points $p^1, p^2 \in U$ such that they can be joined by a $g$-classically causal curve $\gamma\colon [0,1] \to M$ with $\gamma([0,1]) \subset U$.  
    Suppose that there exists an $i \in \{0, \ldots, d\}$ such that $y_i(p^1) > y_i(p^2)$. 
    Then $w:=p^2-p^1 \in \R^{d+1} \setminus C$.
    But by assumption, we have that $\gamma$ is a $m_2$-classically causal curve from $p^1$ to $p^2$, hence as $m_2$ is constant, we get that $p^2-p^1 \in C$ which yields a contradiction. Hence, we have that for all $i$, it holds $y_i(p^1) \leq y_i(p^2)$. As $p^1, p^2$ were arbitrary, this proves the lemma. 
\end{proof}

\smallskip

\begin{lemma}\label{classically_causal_curve_non_constant}
    Let $x \neq z \in M$. Let $\gamma\colon[0,1] \to M$ be a Lipschitz continuous classically causal curve from $x$ to $z$. Then there exists a Lipschitz continuous and classically causal curve $\Tilde{\gamma}\colon[0,1] \to M$ with $\Tilde{\gamma}_0=x, \Tilde{\gamma}_1=z$, such that for almost all $t \in [0,1]$ it holds $\dot{\Tilde{\gamma}}_t \neq 0$, and such that 
    \begin{align*}
        L_g(\gamma) = L_g(\Tilde{\gamma}).
    \end{align*}
\end{lemma}
\begin{proof}
    Lemma \ref{coordinates_causal_curves_monotone} implies that for every point $p \in M$ there exists a neighbourhood $U$ of $p$ and local coordinates $(y_0, \ldots, y_d)$ for $U$ such that the following holds: For any points $p^1 \leq_g p^2 \in U$ that can be joined by a $g$-classically causal curve that is contained in $U$, it holds that $y_i(p^1) \leq y_i(p^2)$ for all $i =0, \ldots, d$.
    We can cover $\gamma([0,1])$ with finitely many of such neighbourhoods and may hence assume that $\gamma([0,1])$ is contained in one such $U$. From now on, we work in $\R^{d+1}$ with the coordinates $y$ and for $p \in U$ and $i =0, \ldots, d$ we write $p_i$ for $y_i(p)$. We may assume that
    \begin{align}\label{l1_unit_dist}
        \sum_{i=0}^d z_i-x_i =1.
    \end{align}
    Define the function $f\colon[0,1] \to [0,1]$ via
    \begin{align*}
        f(t):= \int_0^t \sum_{i=0}^d (\dot{\gamma}_s)_i\, \di s.
    \end{align*}
    Note that $f$ is Lipschitz continuous as $\gamma$ is Lipschitz continuous. Moreover, our choice of coordinates together with \eqref{l1_unit_dist} yields that $f$ is increasing and surjective. In particular, it holds
    \begin{align*}
        f(t) = \sum_{i=0}^d (\gamma_t)_i -(\gamma_0)_i = \sum_{i=0}^d |(\gamma_t)_i -(\gamma_0)_i|.
    \end{align*}
    We now define the curve $\Tilde{\gamma}\colon [0,1] \to U$ via
    \begin{align}
    \Tilde{\gamma}(\tau):= \gamma(\sup\{\sigma: f(\sigma) \leq \tau\}).
    \end{align}
    We immediately get that $\gamma_0=\Tilde{\gamma}_0$ and $\gamma_1=\Tilde{\gamma}_1$.
    Now let $0 \leq \tau_1 < \tau_2 \leq 1$ and for $j=1, 2$ define $\sigma_j:=\sup\{\sigma: f(\sigma) \leq \tau_j\}$. Then note that as $f$ is increasing it holds $\sigma_2 \geq \sigma_1$. It follows that 
    \begin{align}\label{tilde_gamma_causal}
        \Tilde{\gamma}(\tau_1) = \gamma(\sigma_1) \leq_g \gamma(\sigma_2) = \Tilde{\gamma}(\tau_2).
    \end{align}
    Moreover, as $f$ is continuous, we get that $f(\sigma_j)=\tau_j$ for $j=1, 2$. In particular, this yields
    \begin{align*}
        \sum_{i=0}^d|(\Tilde{\gamma}(\tau_2))_i-(\Tilde{\gamma}(\tau_1))_i| = \sum_{i=0}^d(\gamma(\sigma_2))_i-(\gamma(\sigma_1))_i = f(\sigma_2)-f(\sigma_1)=\tau_2-\tau_1. 
    \end{align*}
    It follows that $\Tilde{\gamma}$ has unit $\sfd_{l^1}$-speed in the $y$-coordinates, where $\sfd_{l^1}(p, q):= \sum_{i=0}^d |p_i-q_i|$ for $p, q \in \R^{d+1}$. As all norms on the Euclidean space are equivalent, it follows that $\Tilde{\gamma}$ is $\sfd_{euc}$-bi-Lipschitz and in particular it holds that $\dot{\Tilde{\gamma}}_t \neq 0$ for almost every $t \in [0,1]$. 

        \textbf{Claim 1} $\gamma = \Tilde{\gamma} \circ f$.\\
    \textit{Proof of claim 1.} Let $t \in [0,1]$. Define $t^*:= \sup\{\sigma: f(\sigma) \leq f(t)\}$. 
    As $f$ is continuous it follows that $f(t) = f(t^*)$. By the definition of $f$, it follows that $\gamma(t) = \gamma(t^*)$. Hence, 
    \begin{align*}
        \Tilde{\gamma}(f(t)) = \gamma(t^*) = \gamma(t).
    \end{align*}
    This proves claim 1. 

     \textbf{Claim 2.} $\dot{\Tilde{\gamma}}_t$ is causal and future directed for almost every $t \in [0,1]$. \\
    \textit{Proof of claim 2.}
    By our previous observations we have that $\dot{\Tilde{\gamma}}_t$ exists and is non-zero for almost every $t \in [0,1]$ so we may fix one such $t$.
     Denote $p:= \Tilde{\gamma}_t \in U \subset \R^{d+1}$.  Define the constant Lorentzian metric $g_\e$ on $\R^{d+1}$ via $g_\e(w,w):=g(p)(w,w)-\e g(p)(X(p), w)^2$. Then $g_\e \succ g(p)$ and hence there exists a neighbourhood $U_\e \subset U$ of $p$ such that $g_\e \succ g$ on $U_\e$.
    We argue by contradiction. 
    Suppose that $v:=\dot{\Tilde{\gamma}}_t\in \R^{d+1} \cong T_pM$ was spacelike or causal and past directed. Then $g(p)(v, v) >0$ or $g(p)(v, X(p))>0$.
     
    Note that $\lim_{\e \to 0} g_\e = g(p)$, hence there exists an $\e>0$ such that $g_\e(v,v)>0$ or $g_\e(v, X(p))>0$. Fix such an $\e$ and possibly shrink $U_\e \ni p$ such that the vector $X(p) \in \R^{d+1}$ is $g$-timelike everywhere in $U_\e$.
    Now let $h_* >0$ be small enough such that $\Tilde{\gamma}([t, t+h_*]) \subset U_\e$. Pick $h \in (0, h_*)$.
    Denote by $\sigma_0:= \sup\{\sigma: f(\sigma) \leq t\}$. Then $\gamma_{\sigma_0} = p$. Moreover, denote $\sigma_h:= \sup\{\sigma: f(\sigma) \leq t+h\}$. Then $\gamma_{\sigma_h} = \Tilde{\gamma}_{t+h}$. 
    Note that by claim 1, we get that $\gamma([\sigma_0, \sigma_h]) = \Tilde{\gamma}([t, t+h]) \subset U_\e$. Hence, we can reparametrise $\gamma|_{[\sigma_0, \sigma_h]}$ to obtain a $g$-classically causal curve $\check{\gamma}:[0,1] \to U_\e$ that connects $\Tilde{\gamma}_t$ with $\Tilde{\gamma}_{t+h}$. By our choice of $U_\e$, we get that $\check{\gamma}$ is $g_\e$-classically causal. However, as $g_\e$ is constant and its future cone is a convex set, we get that
    \begin{align*}
        g_\e(\Tilde{\gamma}_{t+h}-\Tilde{\gamma}_t, \Tilde{\gamma}_{t+h}-\Tilde{\gamma}_t) \leq 0 \ \mathrm{and}\ g_\e(\Tilde{\gamma}_{t+h}-\Tilde{\gamma}_t, X(p)) \leq 0,
    \end{align*}
    i.e., $\Tilde{\gamma}_{t+h}-\Tilde{\gamma}_t$ as a vector in $\R^{d+1}$ is $g_\e$-causal and future directed. 
    Note that $v = \lim_{h \to 0} \frac{\Tilde{\gamma}_{t+h}-\Tilde{\gamma}_t}{h}$, so we get that $g_\e(v, v) \leq 0$ and $g_\e(v, X(p)) \leq 0$, which contradicts our assumption. Hence, $v$ has to be causal and future directed. 
    This proves claim 2. 
    
    Note that this shows that $\Tilde{\gamma}$ is  $g$-classically causal.
    
    Now, both $\Tilde{\gamma}$ and $f$ are Lipschitz continuous, hence we may apply the chain rule and get that 
    \begin{align*}
        \dot{\gamma}_{t} = f'(t) \dot{\Tilde{\gamma}}_{f(t)} \ \mathrm{almost\ everywhere}.
    \end{align*}
    It follows that 
    \begin{align*}
        L_g(\gamma) = \int_0^1 \sqrt{g({\gamma}_t)(\dot{{\gamma}}_t, \dot{{\gamma}}_t)} \, \di t &= \int_0^1 f'(t)\sqrt{g(\Tilde{\gamma}(f(t)))(\dot{\Tilde{\gamma}}_{f(t)}, \dot{\Tilde{\gamma}}_{f(t)})} \, \di t.
    \end{align*}
    However, as $f$ is monotone and absolutely continuous, we may apply substitution (see for instance \cite[Theorem 5.42]{leoni2017first}) and get that 
    \begin{align*}
        \int_0^1 f'(t)\sqrt{g(\Tilde{\gamma}(f(t)))(\dot{\Tilde{\gamma}}_{f(t)}, \dot{\Tilde{\gamma}}_{f(t)})} \, \di t = \int_0^1\sqrt{g(\Tilde{\gamma}_\tau)(\dot{\Tilde{\gamma}}_\tau,\dot{\Tilde{\gamma}}_\tau )}\, \di \tau = L_g(\Tilde{\gamma}),
    \end{align*}
    which yields $L_g(\gamma)= L_g(\Tilde{\gamma})$ and hence proves the lemma. 
\end{proof}

Recall that the spacetime $(M, g)$ is said to be 
\begin{itemize}
    \item \textit{strongly causal} if for every $p \in M$ and any neighbourhood $U \subset M$ of $p$, there exists a neighbourhood $V \subset U$ of $p$ such that every classically causal and future directed curve with endpoints in $V$ is contained in $U$.
    \item \textit{causally simple} if there are no non-constant closed classically causal and future directed curves and if $\leq$ is closed in the manifold topology.
    \item \textit{globally hyperbolic} if it is strongly causal and if for any two points $p, q \in M$ the set $J^+(p) \cap J^-(q)$ is compact. 
\end{itemize}

The next result is known for smooth Lorentzian spacetimes (see for instance \cite{minguzzi2019lorentzian}) and the proof works similarly. Alternatively, Minguzzi's result on closed cone structures \cite[Theorem 2.47]{minguzzi2019causality} implies the following lemma.
\smallskip

\begin{lemma}\label{simple_to_strong_causal}
    If $(M,g)$ is causally simple then $(M, g)$ is strongly causal.
\end{lemma}
\begin{proof}
    Suppose for a contradiction that $(M, g)$ is not strongly causal. Then there exists a $p\in M$ and an open set $U \subset M$ such that $p \in U$ and such that for every open set $V\in M$ with $p \in V$ there exist $q^-, q^+ \in V$ and a classically causal and future directed curve $\gamma:[0,1] \to M$ with $\gamma_0=q^-$ and $\gamma_1 = q^+$ that leaves $U$.
    By possibly shrinking $U$ we may assume that $U \ni p$ is relatively compact. For $n \in \N$, fix open sets $V_n$ such that $p \in V_{n+1} \subset V_n \subset U$ and such that $\bigcap_{n \in \N} V_n = \{p\}$.
    Now, by our assumption, there exist $q_n^-, q_n^+ \in V_n$ and a causal curve $\gamma_n:[0,1] \to M$ from $q_n^-$ to $q_n^+$ that leaves $U$. Denote 
    \begin{align*}
        q_n^0:= \gamma\big(\min\{t \in [0,1]: \gamma_t \notin U\}\big) \in \partial U.
    \end{align*}
    By potentially passing to a subsequence, the relative compactness of $U$ and the choice of the $V_n$ yields that 
    \begin{align*}
        \lim_{n \to \infty} q_n^- = \lim_{n \to \infty} q_n^+ = p,
    \end{align*}
    and that there exists a $q \in \partial U$ such that 
    \begin{align*}
        \lim_{n \to \infty} q_n^0 = q.
    \end{align*}
    Moreover, we have that $q_n^- \leq_g q_n^0 \leq_g q_n^+$, hence the fact that $\leq_g$ is closed yields that 
    \begin{align*}
        p \leq_g q \leq_g p.
    \end{align*}
    However, $p \notin \partial U$, hence $q \neq p$. Thus, there exists a non-constant closed classically causal curve, namely starting at $p$ and passing through $q$. This contradicts the causal simplicity and hence proves the lemma. 
\end{proof}
\begin{assumption}
    From now on we will assume $(M, g)$ to be causally simple.
\end{assumption}
\smallskip

\begin{remark}\label{remark_causal_simplicity}
    The assumption of causal simplicity is of technical nature, as causal simplicity is necessary to enter the framework of a Polish metric measure spacetime as in \cite{beran2024nonlinear}. Indeed, it is needed for $\{\ell_g \geq 0\}$ to be closed in the manifold topology. However, strong causality is implied by Lemma \ref{simple_to_strong_causal} and suffices for our proofs. By \cite[Proposition 3.3]{samann2016global}, any globally hyperbolic continuous spacetime is causally simple and strongly causal.
\end{remark}
From now on, we will consider the metric measure spacetime $(M, \tau_h, \ell_g, \vol_g)$.  When speaking of causal paths we refer to maps $\gamma:[0,1] \to M$ such that for each $s, t \in [0,1]$ with $s \leq t$, it holds $\gamma_s \leq_g \gamma_t$ (these are the future directed and causal paths as defined in \cite{beran2024nonlinear}). Note that the notion of a causal path uses the definition of $\ell_g$ based on classically causal curves. When speaking of a causal curve, we refer to continuous causal paths.
\smallskip

A priori, all we know is that the class of classically causal curves is contained in the class of causal paths. We quickly notice the following useful fact.
\smallskip

\begin{lemma}
    Let $(M, g)$  be a spacetime and $V \subset U \subset M$ neighbourhoods such that every classically causal curve with endpoints in $V$ is contained in $U$. Then, every causal path with endpoints in $V$ is contained in $U$.
\end{lemma}
\begin{proof}
    Suppose there exists a causal path $\gamma\colon [0,1] \to M$ with $\gamma_0, \gamma_1 \in V$ and such that there is a $t \in (0,1)$ with $\gamma_t \notin U$. By the definition of causal paths, we have that $\gamma_0 \leq_g \gamma_t \leq_g \gamma_1$. Hence, there exist Lipschitz continuous classically causal curves $\lambda^1, \lambda^2\colon [0,1] \to M$ such that $\lambda^1_0 = \gamma_0$, $\lambda^1_1=\lambda^2_0 = \gamma_t$, and $\lambda^2_1 = \gamma_1$. Then, the concatenation $\lambda$ of $\lambda^1$ and $\lambda^2$ is a classically causal curve with endpoints in $V$ that leaves $U$. This is a contradiction and proves the lemma. 
\end{proof}

This shows that the strong causality assumption also holds for causal paths, not only for classically causal curves. Similarly, the causal simplicity assumption holds for causal paths as well. 
\smallskip

\begin{lemma}\label{criterion_spacelike_vectors}
    Let $p \in M$ and $0 \neq w \in T_pM$ be spacelike. Then there exist timelike vectors $u, v \in T_pM$ such that one of them is future directed and the other one past directed and such that $g( u, w), g(v, w) >0$. 
\end{lemma}
\begin{proof}
    As this is a statement about the tangent space, we may assume to be working in $\R^{d+1}\cong T_pM$. Furthermore, we may assume the coordinates to be chosen in a way that $g(p) = \mathrm{diag}(-1, 1, \ldots, 1)$.
    Write $w = (w_0, \ldots, w_d) \in \R^{d+1}$.
    By applying a rotation, we may assume that $w_2 = \ldots = w_d = 0$. This reduces the problem to the case $d=1$ and $g= \mathrm{diag}(-1, 1)$.\\
    Now as $w$ is spacelike, we get that $w_1^2 > w_0^2$.
    Define the vectors $\vartheta := (w_1, w_0), \vartheta':=-\vartheta= (-w_1, -w_0)$ and note that $\vartheta, \vartheta'$ are both timelike but one is future directed and the other one past directed. Moreover, it holds that $g(w, \vartheta) = g(w, -\vartheta) =0$. Fix an $\alpha >0$ such that $\vartheta + [-\alpha, \alpha]e_1$ is timelike and hence in the same cone (future or past) as $\vartheta$ and such that similarly $\vartheta' + [-\alpha, \alpha]e_1$ is timelike and hence in the same cone as $\vartheta'$.
    Now we distinguish cases:\\
    \textbf{Case 1.} $w_1>0$. \\ 
    Then $u = \vartheta + \alpha e_1$ and $v = \vartheta' + \alpha e_1$ yield the desired. Indeed,
    \begin{align*}
        &g(w,u) = g(w, \vartheta) + \alpha w_1 = \alpha w_1 >0, \\
        &g(w,v) = g(w, \vartheta') + \alpha w_1 = \alpha w_1 >0.
    \end{align*}
    \\
    \textbf{Case 2.} $w_1<0$. \\
    In this case, $u = \vartheta - \alpha e_1$ and $v = \vartheta' - \alpha e_1$ yield the desired.
    This proves the lemma.
\end{proof}
\smallskip

\begin{lemma}\label{cauchy_schwartz_optimiser}
   Let $p \in M$ and fix a local trivialisation that identifies $T_pM$ with $\R^{d+1}$. Let $w \in T_pM$ be causal and past directed and let $\e>0$ and $\beta>0$. Then there exists a timelike and future directed vector $v \in T_pM$ such that $\beta \leq |v|_{euc} \leq 2\beta$ and such that
   \begin{align*}
       \langle v, w\rangle_g \leq \e|v|_g + |v|_g|w|_g.
   \end{align*}
\end{lemma}
\begin{proof}
    If $w$ is timelike then for any $\lambda>0$, we have that $v_\lambda := -\lambda w$ is timelike and future directed and 
    \begin{align*}
        \langle v_\lambda, w\rangle_g =|v_\lambda|_g|w|_g.
    \end{align*}
    Now we can pick $\lambda$ such that $2\beta =|v_\lambda|_{euc}$. 
    
    Now assume that $w$ is null. Write $w=(w_0, \ldots, w_d) \in \R^{d+1}$. Applying linear change $L$ of coordinates, we may again assume that (in the new coordinates after applying $L$) $g(p) = \diag(-1, 1, \ldots, 1)$ and that $w_2= \ldots=w_d=0$. 
    We may furthermore assume that $e_0$ is future directed and that $w=(-a, -a, 0, \ldots, 0)$ for some $a >0$.
    Fix some $b >0$ such that $|L^{-1}(b, b, 0,\ldots, 0)|_{euc}=3\beta/2$. 
    Define the timelike and future directed vector $v_\delta := (b+\delta, b, 0, \ldots, 0)$ for $\delta >0$. We can compute that 
    \begin{align*}
        g(v_\delta,v_\delta) = -2b\delta-\delta^2 \ \mathrm{and}\ |v_\delta|_g = \sqrt{2b\delta+\delta^2}.
    \end{align*}
    Furthermore, we get that 
    \begin{align*}
        \langle v_\delta, w \rangle_g = a \delta.
    \end{align*}
    There exists a $\delta_0 >0$ such that for $\delta \in (0, \delta_0)$, it holds
    \begin{align*}
         \langle v_\delta, w \rangle_g = a \delta \leq \e\sqrt{2b\delta + \delta^2} = \e |v_\delta|_g  =  \e |v_\delta|_g + |w|_g |v_\delta|_g.
    \end{align*}
    There exists a $\delta_1 >0$ such that for $\delta \in (0, \delta_1)$, it holds that $\beta \leq |L^{-1}v_\delta|_{euc} \leq 2\beta$. Then for any $\delta \in (0, \min(\delta_0, \delta_1))$, we get that $v_\delta$ satisfies the desired properties.  
\end{proof}
Many of the following arguments will be local, so we will often work in small neighbourhoods $U \subset M$ that may be assumed to be contained in one coordinate patch. To keep notation short, we will oftentimes assume $U \subset \R^{d+1}$ without explicitly mentioning (pushforwards by) a chart $\psi: U \to \R^{d+1}$.
\smallskip

\begin{lemma}\label{derivative_causal_where_exists}
    Let $\gamma\colon[0,1] \to M$ be a causal path. Then for each $t \in [0,1]$, at which $\gamma$ is differentiable, the derivative $\dot{\gamma}_t \in T_{\gamma_t}M$ is either $0$ or causal and future directed.
\end{lemma}
\begin{proof}
    Fix a $t \in [0,1]$ such that $\gamma$ is differentiable at $t$ and denote $p:= \gamma_t$ and $v := \dot{\gamma}_t$. We may pick local coordinates for an open set $ U \ni p$ such that $p \in \R^{d+1}$ and such that $g(p) = \diag(-1, 1, \ldots, 1)$. 
    For $\e >0$, define the Lorentzian metric $m_\e:= \diag(-1-\e, 1, \ldots, 1)$ on $U$. Then for all $\e >0$, we have that there exists an open neighbourhood $U_\e \subset U$ such that $p \in U_\e$ and such that $m_\e \succ g$ on $U_\e$. Moreover, by the strong causality of $(M, g)$, there exists a neighbourhood $p \in V_\e \subset U_\e$ such that every $g$-causal path with endpoints in $V_\e$ is contained in $U_\e$. 

    We argue by contradiction and assume that $v$ is non-zero and spacelike or causal and past directed. Then $g(p)(v,v) >0$ or $g(p)(X(p), v)>0$. As $\lim_{\e \to 0}m_\e = g(p)$, we have that there exists an $\e>0$ small enough such that $m_\e(v, v) >0$ or $m_\e(X(p), v)>0$. Fix such an $\e>0$. Note that for $h_* >0$ small enough, we have that $\gamma_{t+h} \in V_\e$ for any $h\in (0, h_*)$. For such $h$, define
    \begin{align*}
        v_h:= \frac{\gamma_{t+h}-\gamma_t}{h} \in \R^{d+1}. 
    \end{align*}
    As $\gamma$ is causal, we have that $\gamma_{t+h} = \gamma_t +hv_h \geq_g \gamma_t$, hence there exists a Lipschitz continuous, $g$-classically causal curve $\lambda:[0,1] \to M$ such that $\lambda_0 = \gamma_t$, $\lambda_1 = \gamma_{t+h}$. But then $\lambda$ is contained in $U_\e$. Hence, $\lambda$ is $m_\e$-classically causal. 
    But $m_\e$ is constant on $U_\e$ and the causal future cone of $m_\e$ is a convex cone in $\R^{d+1}$, hence, we get that $v_h$ is $m_\e$-causal and future directed for $h \in (0, h_*)$. 
    Now, we have that $v_h \to v$ as $h \to 0$, so the closedness of the $m_\e$-causal future cone implies that $v$ is $m_\e$-causal and future directed. This contradicts our assumption and hence proves the lemma.
\end{proof}
\begin{remark}
 Lemma \ref{derivative_causal_where_exists} shows that in the setting of a manifold with a continuous Lorentzian metric the class of causal paths is a sensible extension of the class of classically causal curves, in the sense that the derivatives of these maps are causal or zero wherever they exist. This observation is in the flavour of \cite[Proposition 5.9]{kunzinger2018lorentzian}. 
\end{remark}

A similar observation as the following one can be found in \cite{beran2024nonlinear}.
\smallskip 

\begin{proposition}\label{causal_function_gradient}
    Let $f$ be a causal function on $M$. Then for each point $p \in M$ at which $f$ is differentiable, the gradient $\nabla_g f(p) \in T_pM$ is either zero or causal and past directed.
\end{proposition}
\begin{proof}
    Take any timelike and future directed vector $v \in T_pM$. Working in a small neighbourhood $U$ around $p$, we may assume that  $U \subset \R^{d+1}$ and identify the tangent spaces with $\R^{d+1}$. 
    Then, by the continuity of $g$, there exists a neighbourhood $V$ around $p$ such that for each $q \in V$, $g(q)(v,v) <0$, where $v$ is seen as a vector in $\R^{d+1}$. Hence, there exists an $s >0$ such that the curve $\gamma: [-s, s] \to M, t \mapsto p+ t v$ is a classically causal and future directed curve.
    Note that $\gamma$ is differentiable at $t=0$. 
    Then, the chain rule yields that 
    \begin{align*}
        \lim_{h\to 0} \frac{f(\gamma_{h})-f(\gamma_0)}{h} = \di f_p(v) = \langle \nabla_g f, v\rangle_g(p).
    \end{align*}
    The causality of $f$ and $\gamma$ yields that 
    \begin{align*}
        \frac{f(\gamma_{h})-f(\gamma_0)}{h} \geq 0,\  \forall 0\neq h \in [-s, s].
    \end{align*}
    Hence, $\langle \nabla_g f, v\rangle_g(p) \geq 0$. As $v$ was chosen arbitrarily, we get that for all timelike and future directed vectors $v \in T_pM$ it holds
    \begin{align*}
        \langle \nabla_g f, v\rangle_g (p) \geq 0.
    \end{align*}
    Now assume $\nabla_g f(p)$ was non-zero and spacelike. Then by Lemma \ref{criterion_spacelike_vectors}, there exists a past directed timelike vector $u \in T_pM$ such that $\langle \nabla_g f, u\rangle_g (p)>0$. But then the vector $v := -u$ is timelike and future directed and satisfies $\langle \nabla_g f, v\rangle_g (p) < 0$, which is a contradiction. Similarly, considering $v = X(p)$, we get that $\nabla_g f(p)$ cannot be causal and future directed.
\end{proof}
\smallskip

\begin{lemma}\label{causal_derivative_of_causal_curves}
     Let $g$ be continuous and strongly causal, $\gamma\colon[0,1] \to  M$ a causal path. Then for each $t \in [0,1)$ where $\gamma$ is differentiable, it holds that 
     \begin{align*}
         \lim_{h \searrow 0} \frac{\ell_g(\gamma_t, \gamma_{t+h})}{h} =|\dot{\gamma}_t|_g.
     \end{align*}
\end{lemma}
\begin{proof}
    Pick a $t \in [0, 1)$ such that $\gamma$ is differentiable at $t$ and denote $p:= \gamma_t \in M$.
    We may restrict to a small neighbourhood $U$ of $p$ and using local coordinates, we may assume that $U \subset \R^{d+1}$ and $g(p) = \mathrm{diag}(-1, 1, \ldots, 1)$. Furthermore, we may possibly shrink $U$ such that it holds \mbox{$g \prec m_2 := \diag(-2,1, \ldots, 1)$} on $U \ni p$ and such that for every $g$-causal vector $v \in TU$ it holds $m_2(v,v) < g(v,v) \leq 0$. 

    We will first consider the case $\dot{\gamma}_t=0$.
    Now, there exist $p^-, p^+\in U$ such that $p \in U':=  (J_{m_2}^-(p^+) \cap J_{m_2}^+(p^-))^\circ \subset U$ (see \cite[Theorem 2.14]{minguzzi2008causal}). Furthermore, by the strong causality of $g$, we can now pick a neighbourhood $V \subset U'$ of $p$ such that every $g$-causal path with endpoints in $V$ is contained in $U'$.  
    Thus, for any $q \in V$, we have that any $g$-classically causal and future directed curve $\lambda$ from $p$ to $q$ is then contained in $U'$. 
    Hence it is also $m_2$-classically causal, future directed, and $L_g(\lambda) < L_{m_2}(\lambda)$, which yields that for $q \in V$, it holds
    \begin{align}\label{tau_vs_taum_zerocase}
        \ell_g(p, q) \leq \ell_{m_2}(p, q).
    \end{align}
    Moreover, note that for $q \in U'$, it holds
    \begin{align}\label{tau_vs_d_zerocase}
        \ell_{m_2}(p, q) \leq 2  \sfd_{euc}(p, q).
    \end{align}
    Now, as $\gamma$ is differentiable and thus continuous at $t$, we have that there exists a $h_0 > 0$ such that for $h \in (0, h_0)$, it holds $\gamma_{t+h} \in V$.
    Hence, considering only $h \in (0, h_0)$, we can apply \eqref{tau_vs_taum_zerocase} and \eqref{tau_vs_d_zerocase} which when passing to the limit yields that
    \begin{align}
        0 \leq \limsup_{h \searrow 0}\frac{\ell_g(\gamma_t, \gamma_{t+h})}{h} \leq \limsup_{h \searrow 0} \frac{2\sfd_{euc}(\gamma_t, \gamma_{t+h})}{h} = 2|\dot{\gamma}_t|_{euc} = 0. 
    \end{align}
    This finishes the proof in the case $\dot{\gamma}_t=0$. 

    From now on we assume that $\dot{\gamma}_t\neq 0$. Then Lemma \ref{derivative_causal_where_exists} ensures that $\dot{\gamma}_t$ is causal and future directed.
    We will show that 
    \begin{align}\label{limsup_liminf}
         \limsup_{h \searrow 0} \frac{\ell_g(\gamma_t, \gamma_{t+h})}{h} \leq |\dot{\gamma}_t|_g, \ \mathrm{and} \ \liminf_{h \searrow 0} \frac{\ell_g(\gamma_t, \gamma_{t+h})}{h} \geq |\dot{\gamma}_t|_g.
     \end{align}
     Working in the neighbourhood $U$ (as above) of $p$ permits to assume $U \subset \R^{d+1}$ and allows us to assume that $p =0$, $g(p)=\diag(-1, 1, \ldots, 1)$, and that $e_0$ is future directed.
    
    By potentially reparametrising $\gamma$, we may assume that $|\dot{\gamma}_t|_{euc} = 1$.
    Identifying the tangent space $T_pM$ with $\R^{d+1}$ as well, we may think of $v := \dot{\gamma}_t$ as a vector in $\R^{d+1}$. 

     We start with the first inequality of \eqref{limsup_liminf} and prove it by contradiction. Suppose for a contradiction that there was an $\e >0$ and a sequence $h_n \in (0, 1-t)$ with $\lim_{n \to \infty} h_n = 0$ and such that for all $n \in \N$, it holds 
    \begin{align*}
        \frac{\ell_g(\gamma_t, \gamma_{t+h_n})}{h_n} > |\dot{\gamma}_t|_g + \e = |v|_{g(p)} + \e = |v|_{g(0)} + \e.
    \end{align*}

Define 
\begin{align}\label{set_delta_causal_speed}
    \delta := \Big(\frac{\e}{12}\Big)^2.
\end{align}
   
    Define the Lorentzian metric $m_{\delta} := \mathrm{diag}(-1-\delta, 1, \ldots, 1)$ on $U$. Then we can find a neighbourhood $U_\delta \subset U$ of $p=0$ such that for all $q \in U_\delta$ it holds
    \begin{align*}
        &\norm{g(0)-g(q)}_{op} \leq \delta ,\ \norm{g(q)-m_{\delta}}_{op} \leq 2\delta , \ m_\delta \succ g(q).
    \end{align*}
    Moreover, we can shrink $U_\delta \ni 0$ to even get that for each $q \in U_\delta$, $w \in T_qM$ $g$-causal, it holds
    \begin{align}\label{m_delta_stronger_than_g}
        -g(q)(w,w) < - m_{\delta}(w,w).
    \end{align}
    Now define the set $C \subset \R^{d+1}$ as
    \begin{align*}
        C:= \left\{ u=(u_0, \ldots, u_d) \in \R^{d+1}, u_0 \geq 0, (1+\delta)u_0^2 -\Big(\sum_{i=1}^d u_i^2\Big) \geq \Big(|v|_{g(0)} + \frac{\e}{3}\Big)^2\Big(\sum_{i=0}^d u_i^2\Big)  \right\}. 
    \end{align*}
    In the spacetime $(\R^{d+1}, m_\delta)$, these are exactly those points $u$ in the causal future of $0$ that satisfy $\ell_{m_\delta}(0, u)^2 \geq (|v|_{g(0)}+\frac{\e}{3})^2\sfd_{euc}(0, u)^2$.
    Applying \cite[Theorem 2.14]{minguzzi2008causal}, we can find an open neighourhood $U_\delta' \subset U_\delta$ that is $m_\delta$-globally hyperbolic, contains $0$, and that is of the form $0 \in (I^+_{m_\delta}(q^-)\cap I^-_{m_\delta}(q^+) )^\circ = U_\delta'$ for $q^-, q^+ \in U_\delta$. By the strong causality of the spacetime $(M, g)$, we can find a neighbourhood $V_\delta \subset U_\delta'$ of $0$ such that any $g$-causal path between two points in $V_\delta$ is contained in $U_\delta'$.
    
   Now, by the continuity of $\gamma$ at $t$, there exists a $h_*>0$ such that $\gamma([t, t+h_*]) \subset V_\delta$. 
   For $n$ large enough, we have that $h_n \leq h_*$, hence $p_n:= \gamma_{t+h_n} \in V_\delta$. By our assumption, there exist Lipschitz continuous $g$-classically causal curves $\lambda_n\colon [0, 1] \to M$ such that $\lambda_n(0)=p=0$ and $\lambda_n(1)= p_n$ and 
   \begin{align}\label{lambda_n_length}
       L_g(\lambda_n) \geq h_n\Big(\frac{\e}{2} + |v|_{g(0)}\Big).
   \end{align}
   Moreover, we have that $\lambda_n([0,1]) \subset U_\delta'$. %Hence, potentially reparametrising, we may assume that they are $L$-Lipschitz, where $L$ is as above and only depends on $U$ and hence on neither $\delta$ nor $\e$.
   By the definition of $U_\delta$ and $U_\delta'$ we have that (for $n$ large enough) the $\lambda_n$ are $m_{\delta}$-classically causal. Then, using \eqref{m_delta_stronger_than_g} and \eqref{lambda_n_length}, we get that
   \begin{align}\label{tau_delta_lower_bound}
      \ell_{m_\delta}(0, p_n) =\ell_{m_\delta}(p, p_n) \geq L_{m_\delta}(\lambda_n) \geq L_g(\lambda_n) \geq h_n\Big(\frac{\e}{2} + |v|_{g(0)}\Big).
   \end{align}
   Moreover, 
   \begin{align}\label{taylor_distance_p_n_p}
       \sfd_{euc}(0, p_n) =\sfd_{euc}(p, p_n) = h_n|v|_{euc} + o(h_n)= h_n + o(h_n).
   \end{align}
   Now, if thinking of the $p_n = ((p_n)_0, \ldots, (p_n)_d)$ as vectors in $\R^{d+1}$, and recalling that in the Minkowski space the maximising curves are direct straight connections, \eqref{tau_delta_lower_bound}, and \eqref{taylor_distance_p_n_p} we get that for $n$ large enough, it holds
   \begin{align}\label{p_n_is_in_C}
       (1+\delta)(p_n)_0^2 -\Big(\sum_{i=1}^d (p_n)_i^2\Big) = \ell^2_{m_\delta}(0, p_n) &\geq h^2_n\Big(\frac{\e}{2} + |v|_{g(0)}\Big)^2 \nonumber \\
       &\geq \Big(\frac{\e}{3} + |v|_{g(0)}\Big)^2 \sfd_{euc}(0, p_n)^2 = \Big(\frac{\e}{3} + |v|_{g(0)}\Big)^2\sum_{i=0}^d (p_n)_i^2.
   \end{align}
Hence, for $n$ large enough, we have that $p_n \in C$.
Note that
\begin{align}\label{m_delta_of_v_vs_g_v}
    |m_{\delta}(v,v)| \leq |g(0)(v,v)| + 2\delta|v|_{euc}^2 = g(0)(v,v) + 2\delta.
\end{align}
Consider
\begin{align*}
    p_h := \frac{\gamma_{t+h}-\gamma_t}{h} \in \R^{d+1} \cong T_pM.
\end{align*}
As $\lim_{h \searrow 0} p_h = v$, \eqref{set_delta_causal_speed} together with \eqref{m_delta_of_v_vs_g_v} yields that 
\begin{align*}
    \lim_{h \searrow 0} |p_h|_{m_\delta} = |v|_{m_\delta} \leq \sqrt{ |v|_{g(0)}^2 + 2\delta|v|^2_{euc}} \leq |v|_{g(0)} + \frac{\e}{6}. 
\end{align*}
Then recalling that $\lim_{h \searrow 0}|p_h|_{euc} \to |v|_{euc} =  1$, we get that for $h$ small enough, it holds
\begin{align*}
     (1+\delta)(p_h)_0^2 -\Big(\sum_{i=1}^d (p_h)_i^2\Big) \leq \Big(|v|_{g(0)} + \frac{\e}{4}\Big)^2\Big(\sum_{i=0}^d (p_h)_i^2\Big).
\end{align*}
As $p_n = h_np_{h_n}$, this contradicts \eqref{p_n_is_in_C} and hence proves the first part of \eqref{limsup_liminf}. 

We will also prove the second part of \eqref{limsup_liminf} by contradiction. Assume there exists an $\e > 0$ and a sequence $h_n \in (0, 1-t)$ such that $\lim_{n \to \infty} h_n = 0$ and such that for all $n$ it holds
\begin{align*}
        0 \leq \frac{\ell_g(\gamma_t, \gamma_{t+h_n})}{h_n} < |\dot{\gamma}_t|_g - \e = |v|_{g(p)}-\e = |v|_{g(0)}-\e.
\end{align*}
First note that this implies $|v|_{g(0)} \geq \e >0$, i.e., $v$ is timelike at $p=0$. 
By the continuity of $g$, there exists an $r > 0$ such that for each $q \in B^{euc}_r(0)$, it holds
\begin{align}\label{liminf_proof_metrics_close}
    \norm{g(q)-g(0)}_{op} \leq \frac{\e}{4}.
\end{align}
Now for $n$ large enough, we have that $\gamma_{t+h_n} \in B^{euc}_r(0)$. 
Define 
\begin{align*}
    v_n:= \frac{\gamma_{t+h_n}-\gamma_{t}}{h_n} \in \R^{d+1},
\end{align*}
and $\lambda_n\colon [0, 1] \to M, s \mapsto sh_nv_n$. 
For $n$ large enough, we have that $\lambda_n([0,1]) \subset B^{euc}_r(0)$. Now note, $v_n \to v$ in $\R^{d+1}$ as $n \to \infty$, so as $v$ is timelike, we have that for $n$ large enough, $\lambda_n$ is a $g$-classically causal and future directed curve. 
But then 
\begin{align}\label{length_estimate_direct_curves}
    \int_0^1 h_n|v_n|_{g(sh_nv_n)}\, \di s = L_g(\lambda_n) \leq \ell_g(\gamma_t, \gamma_{t+h_n}) < h_n(|v|_{g(0)}-\e).
\end{align}
However, by \eqref{liminf_proof_metrics_close}, we get that for $n$ large enough and all $s \in [0,1]$ it holds
\begin{align*}
    ||v_n|_{g(sh_nv_n)}-|v|_{g(0)}|\leq \frac{\e}{2}.
\end{align*}
This contradicts \eqref{length_estimate_direct_curves} and therewith proves the second part of \eqref{limsup_liminf}, which immediately yields the desired result.
\end{proof}

\begin{lemma}\label{lcc_cont_ae}
   Let $\gamma\colon  [0,1] \to M$ be a left continuous causal path. Then $\gamma$ has at most countably many discontinuities. 
\end{lemma}
The proof of this lemma follows the same strategy as the proof of \cite[Lemma 2.27]{beran2024nonlinear}, with the difference that we immediately assume left continuity of $\gamma$ rather than, as they do in \cite{beran2024nonlinear}, working in a so-called forward metric spacetime.
\begin{proof}
    Recall that we fixed a complete Riemannian metric $h$ on $M$ and denote by $\sfd_h$ the induced Polish distance on $M$. For some $c>0$ we define the set
    \begin{align*}
        B:= \{t \in [0,1]: \limsup_{s \to t} \sfd_h(\gamma_s, \gamma_t)>c\}.
    \end{align*}
    If we can prove that $B$ is countable for any $c >0$, the result follows. Assume for a  contradiction that there exists a $c>0$ such that $B$ is uncountable. 
    
    \textbf{Claim:} There exists a $T \in (0,1]$ such that for all $\e>0$, it holds that $B\cap (T-\e, T]$ is uncountable. \\
    \textit{Proof of the claim.} We argue by contradiction. If the claim was false, then for every $T \in (0, 1]$ there exists an $\e(T)>0$ such that $B \cap (T-\e(T), T]$ is countable. Define 
    \begin{align*}
        a:= \inf\{T \in [0,1]: [T, 1]\cap B \ \mathrm{is\ countable}\}.
    \end{align*}
    We know that $a \leq 1-\e(1)<1$. Assume $a>0$. Then 
    \begin{align*}
        (a, 1] \cap B = \bigcup_{n \in \N}[a+n^{-1}, 1] \cap B \ \mathrm{is\ countable}.
    \end{align*}
    But now, we know that there exists an $\e(a)>0$ such that $B \cap (a-\e(a), a]$ is countable, hence $B \cap [a-\e(a)/2, 1]$ is countable, contradicting that $a$ is infimal with this property. It follows that $a=0$ and hence 
    \begin{align*}
        B \cap [0,1] = \Big(\{0\} \cup \bigcup_{n \in \N} [n^{-1}, 1]\Big)\cap B \ \mathrm{is\ countable}.
    \end{align*}
    This is a contradiction and proves the claim.

    Now fix a $T$ as in the claim. 
    Using that $\gamma$ is left continuous, we get that there exists an $\e>0$ such that
    \begin{align}\label{lcc_close}
     \gamma((T-\e, T]) \subset B^{\sfd_h}_{c/3}(\gamma(T)).   
    \end{align}
    But now, by assumption we have that $(T-\e, T] \cap B$ is uncountable, and hence nonempty, so there exists a $t \in (T-\e, T) \cap B$. By the definition of $B$, there exists an $s \in (T-\e, T)$ such that 
    \begin{align}
        \sfd_h(\gamma_s, \gamma_t) >c. 
    \end{align}
    This contradicts \eqref{lcc_close} and hence proves that $B$ is at most countable, finishing the proof of the lemma.  
\end{proof}
\begin{lemma}\label{lcc_diff_a-e}
      Let $\gamma \in \mathrm{LCC}([0,1], M)$. Then $\gamma$ is differentiable at almost every $t \in [0,1]$.
\end{lemma}
\begin{proof}
     First note that by Lemma \ref{lcc_cont_ae}, $\gamma$ is continuous almost everywhere. Pick a point $s \in [0,1]$ such that $\gamma$ is continuous at $s$. 
    Denote $p:= \gamma_s$ and pick a neighbourhood $U$ around $p$ and coordinates $y$ as in Lemma \ref{coordinates_causal_curves_monotone}.
    By the strong causality, we can find a neighbourhood $W \subset U$ of $p$ such that every curve with endpoints in $W$ is contained in $U$. 
    Now as $\gamma$ is continuous at $s$, there exists an $\e> 0$ such that $\gamma([s-\e, s+\e]) \subset W$. Let $\sigma_1\leq  \sigma_2 \in [s-\e, s+\e]$. 
    By assumption, we have that $\gamma_{\sigma_1} \leq_g \gamma_{\sigma_2}$ hence, they can be joined by a classically causal curve. Moreover, any classically causal curve from $\gamma_{\sigma_1} \in W$ to $\gamma_{\sigma_2} \in W$ must be contained in $U$, hence by our choice of $U$, we get that for all $i =0, \ldots, d$, it holds $y_i(\gamma_{\sigma_1}) \leq y_i(\gamma_{\sigma_2})$.  
    Thus, we get that for every continuity point $s \in [0,1]$ of $\gamma$, there exist coordinates $y$ and an $\e>0$ such that  $y \circ \gamma$ is increasing in every variable on $[s-\e, s+\e]$. Hence, for every continuity point $s \in [0,1]$, there exists an $\e>0$ such that $\gamma$ is differentiable almost everywhere in $[s-\e, s+\e]$.

    Now Lemma \ref{null_set_local_to_global} yields that $\gamma$ is almost everywhere differentiable.
\end{proof}
This leads us to the first important result of this paper, stating that the causal speed can be identified as the causal length of the derivative. This can be seen as a Lorentzian version of \cite[Corollary 3.8]{burtscher2012length}.

\smallskip
\begin{corollary}\label{causal_speed_computed}
    Let $\gamma \in \mathrm{LCC}([0,1], M)$. Then for almost every $t \in [0,1]$, the causal speed is given by 
    \begin{align}
        |\dot{\gamma}_t|= |\dot{\gamma}_t|_g.
    \end{align}
\end{corollary}
\begin{proof}
    By Lemma \ref{lcc_diff_a-e}, we have that $\gamma$ is differentiable almost everywhere. Then by Lemma \ref{causal_derivative_of_causal_curves}, we have that for almost every $t \in [0,1]$, it holds
    \begin{align}
        \lim_{h \searrow 0} \frac{\ell_g(\gamma_t, \gamma_{t+h})}{h} = |\dot{\gamma}_t|_g.
    \end{align}
    To conclude, we use Definition \ref{def_causal_speed} (see also \cite[(2.28)]{beran2024nonlinear}), which states that for almost every $t$, it holds
    \begin{align*}
        \lim_{h \searrow 0} \frac{\ell_g(\gamma_t, \gamma_{t+h})}{h} = |\dot{\gamma}_t|.
    \end{align*}
\end{proof} 
\smallskip

\begin{lemma}\label{visibility}
    Let $M$ be a manifold and $g$ be a continuous Lorentzian metric on $M$. Then $\mathrm{Vis}(M) =M$. 
\end{lemma}
\begin{proof}
    Let $p \in M$. Fix an open neighbourhood $U \subset M$ of $p$ and suppose from now on that $U \subset \R^{d+1}$. Pick a vector $v \in \R^{d+1}$ with $g(p)(v,v)<-\alpha<0$ and such that $v$ is future directed. We may assume that $|v|_{euc}=1$.  
    Then, there exists a neighbourhood $V \subset U$ of $p$ such that $v$ is future directed and $g(q)(v,v)<-\frac{\alpha}{2}<0$ at every point $q \in V$. Let $\theta >0$ such that $B_{4\theta}^{\R^{d+1}}(p) \subset V$. 
    For $y \in B_\theta(p)$, define the causal curve \mbox{$\gamma^y\colon [0,1] \to V, t \mapsto y+\theta(2t-1)v$}. Then $\gamma_y([0,1]) \subset B_{4\theta}(p)$.
    Consider the test plan $\boldsymbol{\pi} \in \mathcal{P}(\mathrm{LCC}([0,1], M))$ given by  
    \begin{align*}
        \di \boldsymbol{\pi}(\gamma):= \left\{ \begin{array}{ll}
            \frac{1}{\mathcal{L}^{d+1}(B_\theta(p))}\di \mathcal{L}^{d+1}(y)  & \mathrm{if}\ \gamma=\gamma^y \ \mathrm{for\ some\ } y \in B_\theta(p), \\
            0 & \mathrm{otherwise.}
        \end{array} \right.
    \end{align*}
   By Corollary \ref{causal_speed_computed}, we get that
     \begin{align*}
        |\dot{\gamma}_t|\di \boldsymbol{\pi}(\gamma)\di t\geq  \left\{ \begin{array}{ll}
            \frac{\sqrt{\alpha}\theta}{2\mathcal{L}^{d+1}(B_\theta(p))}\di \mathcal{L}^{d+1}(y)\di t  & \mathrm{if}\ \gamma=\gamma^y \ \mathrm{for\ some\ } y \in B_\theta(p), \\
            0 & \mathrm{otherwise.}
        \end{array} \right.
    \end{align*}
    Then 
    \begin{align*}
        \di (e_t)_\#|\dot{\gamma}_t|\boldsymbol{\pi} \geq  \frac{\sqrt{\alpha}\theta}{2\mathcal{L}^{d+1}(B_\theta(p))} \mathbbm{1}_{B_\theta(p+\theta(2t-1)v)} \di \mathcal{L}^{d+1},
    \end{align*}
    where $\mathbbm{1}_{B_\theta(p+\theta(2t-1)v)}$ denotes the characteristic funcion of the set $B_\theta(p+\theta(2t-1)v)$.
    Hence, writing $(\mathsf{e}(t, \gamma))_\#|\dot{\gamma}_t|\boldsymbol{\pi}= \rho_{\boldsymbol{\pi}}\di \mathcal{L}^{d+1}$, we get that for every $q \in B_{\theta/8}(p)$, it holds
    \begin{align*}
        \rho_{\boldsymbol{\pi}}(q) \geq \frac{\sqrt{\alpha}\theta}{16\mathcal{L}^{d+1}(B_\theta(p))} >0.
    \end{align*}
    Hence, $p \in \mathrm{Vis}(M)$. As $p$ was arbitrary, this proves the lemma. 
\end{proof}

The next proposition has been proved for the smooth case in \cite[Theorem 1.19]{minguzzi2019lorentzian} and \cite[Theorem A.2]{beran2024nonlinear}. We will closely follow their argument.
\smallskip

\begin{lemma}\label{bounded_variation_neighbourhoods}
    Let $f$ be a causal function. 
    Then for almost every $p \in \{|f|<\infty\}$, there exists a neighbourhood $W$ of $p$ such that the distributional derivative of $f$ restricted to $W$ is given by a co-vector valued, finite Radon measure $\mathcal{D}f$.
    In particular,  $f$ is differentiable almost everywhere on $\{|f|<\infty\}$.
\end{lemma}
\begin{proof}
    Lemma \ref{local_finiteness} and Lemma \ref{visibility} yield that for almost every $p\in \mathrm{Vis}(M)\cap \{|f|<\infty\} = \{|f|<\infty\}$, there exists an open neighbourhood $U$ of $p$ such that $f$ is finite on $U$. We may assume that $U \subset \R^{d+1}$ and that $U$ is relatively compact. We may furthermore assume that $g(p)= \mathrm{diag}(-1, 1, \ldots, 1)$. 
    Now we can find an open neighbourhood $ p \in V\subset U$ such that $m_{1/2}:=\mathrm{diag}(-1, 2, \ldots, 2) \prec g$ on $V$. There exists a basis $b_0, \ldots, b_d$ of $\R^{d+1}$ such that the cone $C:=\{\sum_{i=0}^d \alpha_ib_i | \alpha_0, \ldots, \alpha_d \geq 0\}$ is contained in the causal future cone of $m_{1/2}$. Then, $C$ is contained in the $g$-causal future cone of $T_qM$ for all $q \in V$. Now there exists a linear transformation into coordinates $y = (y_0, \ldots, y_d)$ for $ U \subset \R^{d+1}$ such that $\partial_{y_i} = b_i$ for all $0 \leq i \leq d$.
   % As a last step, we pick an open neighbourhood $W$ of $p$ such that $W \subset V$ and $W$ is $g$-globally hyperbolic.
    In these coordinates, we have that for each $x, z\in V$ with $y_i(x) \leq y_i(z)$ for  all $0 \leq i \leq d$, it holds that $x \leq_g z$. 
    Hence the function $f$ is monotone in each variable, which by \cite[Theorem 4]{chabrillac2009continuity} yields $\mathcal{L}^{d+1}$-measurability (and hence $\vol_g$-measurability). We may assume that in $y$-coordinates, we have that $p=0$ and that $ W:= (-a, a)^{d+1} \subset [-a, a]^{d+1} \subset V$ for some $a >0$. Now for each $\beta_0, \ldots, \beta_d \in [-a, a]$, and all $i \in \{0, \ldots, d\}$, it holds
    \begin{align*}
        t \mapsto f(\beta_0, \ldots, \beta_{i-1}, t, \beta_{i+1}, \ldots, \beta_d), \quad t \in [-a, a],
    \end{align*}
    is a monotone (and hence a bounded variation function in one variable) whose distributional differential is bounded above by $f(a, \ldots, a)-f(-a, \ldots, -a) < \infty$. By \cite[Remark 3.104]{ambrosio2000functions}, this proves that $f$ is locally a bounded variation function, i.e., its distributional derivative is a finite co-vector valued Radon measure $\mathcal{D}f$ on $[-a, a]^{d+1}$. Furthermore, \cite[Theorem 14]{chabrillac2009continuity} states that $f$ is $\mathcal{L}^{d+1}$-almost everywhere differentiable in $[-a, a]^{d+1}$. 
\end{proof}
Before stating the next lemma, it is worth understanding the distributional derivative $\mathcal{D}f$ on $\{|f|<\infty\}$ a little better. Fix $p \in \{|f|<\infty\}$ as in Lemma \ref{bounded_variation_neighbourhoods} and choose an open neighbourhood $W$ of $p$ accordingly. Assume that $W \subset \R^{d+1}$. Fix a vector $v \in \R^{d+1}$ such that $g(p)(v,v)<0$ and shrink $W$ to a neighbourhood $U \ni p$ such that $v$ is future directed and timelike everywhere in $U$. Then, as $\mathcal{D}f$ is the distributional derivative of $f$, we have that for every $\varphi \in C_c^\infty(U)$, it holds
\begin{align*}
    \int_U \varphi \, \di \mathcal{D}f(v) = \lim_{h \searrow 0} \int_U \p(y) \frac{f(y+hv)-f(y)}{h} \, \di \mathcal{L}^{d+1}(y).
\end{align*}
This yields that $\mathcal{D}f(v)$ is the limit measure of the sequence $\frac{f(y+hv)-f(y)}{h} \di \mathcal{L}^{d+1}(y)$, as $h \searrow 0$, in the narrow topology.

\smallskip

\begin{lemma}\label{fundamental_thm_of_calculus_for_measure}
    Let $p, v, U$ be as above.
    Given $q \in U$, $\delta, s >0$ such that $B_{2\delta+2s}(q) \subset U$, it holds that 
    \begin{align*}
        \int_{-s}^s \int_{\overline{B_\delta(q)}} \, \di\mathcal{D}f(v)(y+ tv) \di t \geq \int_{\overline{B_\delta(q)}} f(y+sv)-f(y-sv)\, \di \mathcal{L}^{d+1}(y).
    \end{align*}
\end{lemma}
\begin{proof}
    By the upper semi-continuity on closed sets under the narrow convergence of measures, and using a telescoping argument, we get that 
    \begin{align*}
        \int_{-s}^s \int_{\overline{B_\delta(q)}} \, \di\mathcal{D}f(v)(y+ tv) \di t &\geq \limsup_{n \to \infty} \int_{-s}^s \int_{\overline{B_\delta(q)}} \frac{n}{2s}\Big(f\Big(y+tv+\frac{2s}{n}v\Big)-f(y+tv)\Big)\, \di \mathcal{L}^{d+1}(y) \di t \\
        &= \limsup_{n \to \infty}  \int_{\overline{B_\delta(q)}} \int_{-s}^s \frac{n}{2s}\Big(f\Big(y+tv+\frac{2s}{n}v\Big)-f(y+tv)\Big) \, \di t\di \mathcal{L}^{d+1}(y) \\
        & = \limsup_{n \to \infty}  \int_{\overline{B_\delta(q)}} \frac{n}{2s}\int_{[0, \frac{2s}{n})}  f\Big(y+\Big(t+s+\frac{2s}{n}\Big)v\Big)-f(y+(t-s)v) \, \di t\di \mathcal{L}^{d+1}(y)\\
        & \geq  \int_{\overline{B_\delta(q)}} f(y+sv)-f(y-sv) \, \di \mathcal{L}^{d+1}(y),
    \end{align*}
    where the last inequality follows from the fact that $f$ is increasing in $v$-direction.
\end{proof}

\begin{lemma}\label{identify_derivative_with_absolutely_continuous_part}
    Given a causal function $f$ on $(M, g)$ and a neighbourhood $W$ as in Lemma \ref{bounded_variation_neighbourhoods}. Suppose that $W \subset\R^{d+1}$. Denote by $\mathcal{D}f$ the distributional derivative of $f$ in $W$ and split it into its absolutely continuous part with respect to $\mathcal{L}^{d+1}$ denoted by $\mathcal{D}f^a = Df \mathcal{L}^{d+1}$ and its singular part $\mathcal{D}f^s$. Then for almost every $p \in W$, we have that if $f$ is differentiable at $p$, it holds that $\di f(p) = Df(p)$.
\end{lemma}
\begin{proof}
    By \cite[Theorem 3.83]{ambrosio2000functions}, we have that for almost every $p$ it holds that $f$ is approximately differentiable at $p$ and the approximate differential equals the density $Df$ of the absolutely continuous part $\mathcal{D}f^a$ of the distributional derivative of $f$. Hence, this holds for almost every $p$, which satisfies that $f$ is differentiable at $p$. But then the approximate differential and the differential are equal at such points, hence $Df = \di f$ almost everywhere.
\end{proof}

From now on, when writing $\di f$ we refer to the absolutely continuous part of the distributional derivative of $f$, which is a covector valued measurable function that is almost everywhere defined and locally integrable in $\{|f| < \infty\}$. 
The following lemma has been proved in Riemannian signature in \cite{mondino2025equivalence}.

\smallskip
\begin{lemma}\label{Lebesguepointsadvanced_lor} 
Let $V \subset M$ be an open set that lies in one coordinate patch.
Let $k \in L^1_{\rm{loc}}(V, \vol_g)$ and assume that $V \subset \R^{d+1}$. Let $x \in V$  and $\theta > 0$, such that $B_{4\theta}^{euc}(x) \subset V$. Let $\Gamma \subset B_\theta^{euc}(0)$ be a countable set.
For $\delta \in (0, \theta)\cap \Q$, $\dot{\gamma} \in \Gamma$,  we define 
    \begin{align*}
        F^k_{\dot{\gamma}, \delta, x}\colon  t \mapsto \frac{1}{\mathcal{L}^{d+1}(\overline{B_\delta(x)})}\int_{\overline{B_\delta(x)}} |\dot{\gamma}|_{g(y + t\dot{\gamma})}k(y + t\dot{\gamma})\,\di\mathcal{L}^{d+1}(y), \quad  t \in (-1, 1).
    \end{align*}
    Then for $\mathcal{L}^{d+1}$-a.e.\;$x \in V$ and all {$\delta \in (0, \theta)\cap \Q$}, $\dot{\gamma} \in \Gamma$, we have that $t=0$ is a Lebesgue point of $ F^k_{\dot{\gamma}, \delta, x}$.
\end{lemma}
Here, $|\dot{\gamma}|_g$ should be interpreted as $\sqrt{|g(\dot{\gamma}, \dot{\gamma})|}$, as we have not made any assumptions on whether $\dot{\gamma}$ is causal.
 \begin{proof}
     As $g$ is continuous, we get that $k \in L^1_{\rm{loc}}(V, \mathcal{L}^{d+1})$.
     Fix $w \in V$, and choose $\theta > 0$ accordingly. Moreover, fix $\delta$ and $\dot{\gamma}$. For any $z \in B^{euc}_\theta(w)$, the Lebesgue differentiation theorem yields that $\mathcal{L}^1$-almost every $t \in (-1, 1)$ is a Lebesgue point of $ F^k_{\dot{\gamma}, \delta, z}$. 
     In other words that means that for $\mathcal{L}^1$-almost every $y$ on the line $\{z+t\dot{\gamma}, |t| <1\}$, $t=0$ is a Lebesgue point of $ F^k_{\dot{\gamma}, \delta, y}$. The set of non-Lebesgue points on that line shall be denoted by $N_{\dot{\gamma}, \delta, z}$, and it holds that $\mathcal{L}^1(N_{\dot{\gamma}, \delta, z})=0$. Let $H_{\dot{\gamma}, w}$ be the $d$-dimensional hyperplane in $\R^{d+1}$ through $w$ that is orthogonal to $\dot{\gamma}$, and intersected with $B^{euc}_{\theta}(w)$. Then, by Fubini's Theorem, it holds
     \begin{align*}
         \mathcal{L}^{d+1}\Big(\bigcup_{z \in H_{\dot{\gamma}, w}}N_{\dot{\gamma}, \delta, z}\Big) = \int_{ H_{\dot{\gamma}, w}} \mathcal{L}^1(N_{\dot{\gamma}, \delta, z}) \, \di\mathcal{L}^{d}(z) = 0. 
     \end{align*}
     Hence for $\mathcal{L}^{d+1}$-almost every $y \in B^{euc}_\theta(w)$, $t=0$ is a Lebesgue point of $ F^k_{\dot{\gamma}, \delta, y}$. Note that, by construction,   $\dot{\gamma}$ and $\delta$  vary in a countable set. This shows the claim in an open neighbourhood around $w$. As $w$ was arbitrary, the proof is complete by Lemma \ref{null_set_local_to_global}. 
 \end{proof}
Under the same assumptions on $k, \theta, \delta, \Gamma$, Lemma \ref{Lebesguepointsadvanced_lor} holds analogously for 
 \begin{align}\label{def_f_tilde_lebesgue_advanced}
        \Tilde{F}^k_{\dot{\gamma}, \delta, x}\colon  t \mapsto \frac{1}{\mathcal{L}^{d+1}(\overline{B_\delta(x)})}\int_{\overline{B_\delta(x)}} k(y + t\dot{\gamma})\,\di\mathcal{L}^{d+1}(y), \quad  t \in (-1, 1).
    \end{align}

The following proposition describes the maximal weak subslope of a causal function $f$. This is the last ingredient we need to prove infinitesimal Minkowskianity.
 \smallskip
 
\begin{proposition}\label{weak_subslope_identified}
    Let $g$ be continuous and strongly causal and $f\colon M \to \overline{\R}$ be a causal function. Then, the maximal weak subslope $|\di f|$ is given by 
    \begin{align}
        |\di f|(p) =G(p):= \left\{\begin{array}{cc}
          |\nabla_g f|_g(p)   & \mathrm{if}\  |f(p)| < \infty\ \mathrm{and\ } f \ \mathrm{is}\ \mathrm{differentiable\ at\ }p, \\
           +\infty  & \mathrm{otherwise}.
        \end{array} \right.
    \end{align}
\end{proposition}
\begin{remark}
This proposition can be seen as a Lorentzian counterpart to \cite[Proposition 4.24]{mondino2025equivalence}. The proof essentially follows a similar strategy, with the additional difficulty that $f$ is not necessarily $C^1$ and hence the distributional derivative might be measure valued. In the Riemannian case, proving the analogue statement for $C^1$-functions sufficed to identify the synthetic Sobolev space as the classical one (see \cite[Section 4.4]{mondino2025equivalence}), as the theory from \cite{ambrosio2014inventio, ambrosio2014duke} allowed to pass from $C^1$ to all Sobolev functions.  
\end{remark}
Before giving the proof, we summarise the strategy: We start with verifying that $G$ is a weak subslope. Then, to prove maximality, we argue by contradiction and assume that $|\di f|>G$ on a set of positive measure: 
\begin{itemize}
    \item We pick a point $p$ in $\{|\di f|>G\}$ at which $f$ is differentiable and find a timelike vector $v \in T_pM$ such that $G(p)|v|_g =|\nabla_g f|_g(p) |v|_g \approx \langle\nabla_gf(p), v\rangle_g = (f \circ \gamma)'(0) = \lim_{h \searrow 0} \frac{f(p+hv)-f(p-hv)}{2h}$, where $\gamma\colon t \mapsto p+tv$. Recalling Definition \ref{def:subslope}, this suggests that $|\di f| \leq G$ along the curve $\gamma|_{[-h, h]}$ for small $h$.
    \item To make this idea rigorous, we build a test plan of short parallel curves in $v$-direction around $p$ and use Lemma \ref{singular_measure_lebesgue_point} to localise away from the support of the singular part $\mathcal{D}f^s$ of the distributional derivative of $\mathcal{D}f$ of $f$.
\end{itemize}
\begin{proof}[Proof of Proposition \ref{weak_subslope_identified}]
    We start by proving that $G$ is a weak subslope. Fix an arbitrary test plan $\boldsymbol{\pi} \in \mathcal{P}(\mathrm{LCC}([0,1], M))$. Moreover, let $N \subset M$ be a $\vol_g$-null set such that $f$ is differentiable or infinite on $M \setminus N$, which exists by Lemma \ref{bounded_variation_neighbourhoods}. 
    As in \cite[Lemma 3.5]{ryborz2025infinitesimal}, we have that for $\boldsymbol{\pi}$-almost every path $\gamma$, it holds that
    \begin{align}\label{transverse_curve}
        \mathcal{L}^1(\gamma^{-1}(N))=0.
    \end{align}
     Consider a path $\gamma \in \mathrm{LCC}([0,1], M)$ such that \eqref{transverse_curve} holds. We distinguish two cases. 

     \textbf{Case 1.} $f(\gamma([0,1]))\subset \R$. \\
     Now, $f$ is differentiable $\mathcal{H}^1$-almost everywhere on $\gamma([0,1])$ and the chain rule yields that for almost every $t \in [0,1]$, it holds
     \begin{align}
         (f \circ \gamma)'(t) = \di f_{\gamma_t}(\dot{\gamma_t}) \geq |\nabla_g f|_g (\gamma_t)|\dot{\gamma}_t|_g,
     \end{align}
     where the last inequality follows from the reverse Cauchy-Schwartz inequality. Note, that we can apply it because $\nabla_g f$ is zero or causal and past directed by Lemma \ref{causal_function_gradient} and $\dot{\gamma}_t$ is zero or causal and future directed by Lemma \ref{derivative_causal_where_exists}. Now, as $f \circ \gamma\colon [0,1] \to \R$ is a monotone function, and $(f \circ \gamma)'(t)$ the absolutely continuous part of the distributional derivative of $f\circ \gamma$, this yields 
     \begin{align}
         f(\gamma_1)-f(\gamma_0) \geq \int_0^1 (f \circ \gamma)'(t) \, \di t \geq \int_0^1 |\nabla f|_g(\gamma_t) |\dot{\gamma}_t|_g \, \di t = \int_0^1 G(\gamma_t) |\dot{\gamma}_t| \, \di t.
     \end{align}

     \textbf{Case 2.} $f(\gamma([0,1])) \not \subset \R$. \\
     Then, as $f$ is causal, we get that $f(\gamma_0) = - \infty$ or $f(\gamma_1) = +\infty$. In both cases, we have that 
     \begin{align*}
         f(\gamma_1)-f(\gamma_0) = +\infty,
     \end{align*}
     hence $G$ is a weak subslope. 

     Now, we want to verify that $G$ is the maximal weak subslope. Suppose that was not the case. Then there exists an $\e > 0$ such that 
     \begin{align*}
         \vol_g(\{|\di f| \geq G+\e\}) >0.
     \end{align*}
     Denote $S:= \{|\di f| \geq G+\e\}$. Then, $G$ is finite at $\vol_g$-almost every point in $S$, hence, by the definition of $G$, $f$ is finite at almost every point of $S$.
     We may assume that $S$ lies inside one coordinate patch and will from now on work in $\R^{d+1}$.     
     Define $\Gamma$ to be a countable and dense set in $\partial B_1^{euc}(0) \subset \R^{d+1}$.
     There exists a $\zeta>1$ such that the set
     \begin{align}\label{bounded_gradient_domain}
         \mathcal{L}^{d+1}(\{|\nabla_g f|_g \leq \zeta\}\cap S)>0. 
     \end{align}
    By Lemma \ref{lebesgue_point_line}, Lemma \ref{singular_measure_lebesgue_point}, Lemma \ref{Lebesguepointsadvanced_lor}, and the Lebesgue differentiation theorem, we can choose a point $p \in S$ and an open neighbourhood $W$ of $p$ such that 
    \begin{itemize}
        \item $p$ is a $\mathcal{L}^{d+1}$-density point of $S \cap \{|\nabla_g f|_g \leq \zeta\}$ and a Lebesgue point of $G, \di f, \min(|\di f|-G, \e)$.
        \item  $f$ is of bounded variation and $f$ is differentiable almost everywhere in $W$ (see Lemma \ref{bounded_variation_neighbourhoods} and \cite{ambrosio2000functions}).
        \item $W$ is relatively compact and $W \subset \{|f|< \infty\}$ (see Lemma \ref{local_finiteness}).
        \item there exists a $\theta>0$ such that for all $\delta\in (0, \theta)\cap \Q$, $\dot{\gamma} \in \theta \Gamma$, it holds that $t=0$ is a Lebesgue point of $F^k_{\dot{\gamma}, \delta, p}$ and $\Tilde{F}^k_{\dot{\gamma}, \delta, p}$, as they are defined in Lemma \ref{Lebesguepointsadvanced_lor} and \eqref{def_f_tilde_lebesgue_advanced} for $k = \di f, G, \min(|\di f|-G, \e)$. In the case of $\di f$ this shall be understood component wise.
        \item $\lim_{\delta\to 0} \frac{\mathcal{D}f^s(B_\delta(p))}{\delta^{d+1}} =0$ (see Lemma \ref{singular_measure_lebesgue_point}). 
        \item for every $\dot{\gamma} \in \Gamma$ and $ \delta \in (0,1)\cap \Q$ small enough, $t=0$ is a Lebesgue point of $L^{E_{\mathcal{D}f, \delta}}_{p,\dot{\gamma}}$ as defined in Lemma \ref{lebesgue_point_line} and Lemma \ref{ball_evaluation_measurable}.
    \end{itemize}
    Define
    \begin{align}\label{choice_of_epsilon}
        \epsilon := \frac{\e}{200(1+\zeta)}. 
    \end{align}
    Now, note that by Lemma \ref{causal_function_gradient}, $\nabla_g f(p) \in T_pM$ is causal and past directed or $0$. If  $\nabla_g f(p)\neq 0$, Lemma \ref{cauchy_schwartz_optimiser} yields that there exists a timelike and future directed vector $\Tilde{v} \in \R^{d+1} \cong T_pM$, such that $\frac{1}{2} \leq|\Tilde{v}|_{euc} \leq 2$ and
    \begin{align}
        \langle \nabla_g f, \Tilde{v} \rangle_g(p) \leq (\epsilon + |\nabla_g f|_g )|\Tilde{v}|_g(p).
    \end{align}
    Now, from our choice of $\Gamma$, we can approximate $\Tilde{v}|\Tilde{v}|_{euc}^{-1}$ to get a timelike and future directed vector $v \in \Gamma$ such that
    \begin{align}\label{almost_equality_vector_v}
        |\nabla_g f|_g |v|_g(p) \leq \langle \nabla_g f, v \rangle_g(p) \leq (3\epsilon + |\nabla_g f|_g) |v|_g(p).
    \end{align}
    If $\nabla_gf(p)=0$, then any timelike and future directed $v \in \Gamma$ will satisfy \eqref{almost_equality_vector_v}, so we may fix an arbitrary one.

    Using a linear transformation $A \in SO(d+1)$ and a translation, we may assume that $v= e_0$ and $p = 0$. Then, by the continuity of $g$, we have that $e_0$ is timelike and future directed in a neighbourhood of $p=0$, so by possibly shrinking $W \ni p=0$, we get that $e_0$ is timelike and future directed in $W$ and that there exists a $\lambda \in (0,1)$ such that for all $q \in W$, it holds
    \begin{align}\label{lambda_minimal_lor_length_e_0}
        |e_0|_{g(q)} \geq \lambda.
    \end{align}
    Now, we can choose $\delta, s \in (0,1)\cap \Q$ such that $B_{2\delta+2s}(0) \subset W$ and such that for $k \in \{G, \di f, \min(\e, |\di f|-G)\}$, it holds  
    \begin{align}
       &\frac{1}{\mathcal{L}^{d+1}(\overline{B_\delta(0)})}\int_{\overline{B_\delta(0)}} |k(y)-k(0)| \, \di \mathcal{L}^{d+1}(y) \leq \lambda \epsilon, \label{lebesgue_for_k} \\
       &\frac{1}{\mathcal{L}^{d+1}(\overline{B_\delta(0)})} \Bigg|\frac{1}{2s}\int_{-s}^s\int_{\overline{B_\delta(0)}} k(y+te_0)|e_0|_{g(y+te_0)} \, \di \mathcal{L}^{d+1}(y) \di t - \int_{\overline{B_\delta(0)}} k(y)|e_0|_{g(y)} \, \di \mathcal{L}^{d+1}(y) \Bigg| \leq \lambda \epsilon, \label{lebesgue_advanced_for_k}
    \end{align}
as well as 
 \begin{align}
       &\frac{1}{\mathcal{L}^{d+1}(\overline{B_\delta(0)})}\int_{\overline{B_\delta(0)}} |\langle \nabla_g f, e_0\rangle_g(y)-\langle \nabla_g f, e_0\rangle_g(0)| \, \di \mathcal{L}^{d+1}(y) \leq \lambda \epsilon, \label{lebesgue_inner_prod} \\
       &\frac{1}{\mathcal{L}^{d+1}(\overline{B_\delta(0)})} \Bigg|\frac{1}{2s}\int_{-s}^s\int_{\overline{B_\delta(0)}} \langle \nabla_g f, e_0\rangle_g(y+te_0) \, \di \mathcal{L}^{d+1}(y) \di t - \int_{\overline{B_\delta(0)}} \langle \nabla_g f, e_0\rangle_g(y) \, \di \mathcal{L}^{d+1}(y) \Bigg| \leq \lambda \epsilon, \label{lebesgue_inner_prod_advanced}
    \end{align}
and  
\begin{align}\label{lebesgue_assumption_for_dfs}
    \frac{|\mathcal{D}f^s(\overline{B_\delta(0)})|}{\mathcal{L}^{d+1}(\overline{B_\delta(0)})} \leq \lambda
    \epsilon\  \mathrm{and}\ \frac{1}{\mathcal{L}^{d+1}(\overline{B_\delta(0)})}\int_{-s}^s \int_{\overline{B_\delta(0+se_0)}} \, \di \mathcal{D}f^s(e_0)\,  \di t \leq 2s\lambda\epsilon. 
\end{align}
For $y \in \overline{B_\delta(p)} = \overline{B_\delta(0)}$ define the causal curve 
\begin{align}
    \gamma^y\colon  [0,1] \to W, t \mapsto y + 2s \Big(t-\frac{1}{2}\Big)  v = y + 2s \Big(t-\frac{1}{2}\Big)e_0.
\end{align}
Now, define the test plan $\boldsymbol{\pi}$ via 
\begin{align}
    \di \boldsymbol{\pi}(\gamma) := \left\{ \begin{array}{ll}
       \frac{1}{\mathcal{L}^{d+1}(\overline{B_\delta(0)})} \di \mathcal{L}^{d+1}(y)   & \mathrm{if}\ \gamma = \gamma^y\ \mathrm{for\ some\ } y \in \overline{B_\delta(0)},  \\
       0  & \mathrm{otherwise}. 
    \end{array} \right.
\end{align}
Using Corollary \ref{causal_speed_computed}, we get that 
\begin{align}\label{evaluate_test_plan_at_G}
    \int \int_0^1 G(\gamma_t)|\dot{\gamma}_t| \, \di t \di \boldsymbol{\pi}(\gamma) &= \frac{1}{\mathcal{L}^{d+1}(\overline{B_\delta(0)})} \int_{\overline{B_\delta(0)}}  \int_0^1 |\nabla_g f|_g\Big(y +2s \Big(t-\frac{1}{2}\Big)  e_0 \Big) 2s|e_0|_g \, \di t \di \mathcal{L}^{d+1}(y)\nonumber \\
    &= \frac{1}{\mathcal{L}^{d+1}(\overline{B_\delta(0)})} \int_{\overline{B_\delta(0)}}  \int_{-s}^s |\nabla_g f|_g(y +te_0)|e_0|_g \, \di t \di \mathcal{L}^{d+1}(y).
\end{align}
Using \eqref{almost_equality_vector_v} and \eqref{lambda_minimal_lor_length_e_0}, combined with \eqref{lebesgue_for_k} - \eqref{lebesgue_inner_prod_advanced} with $k = G$ we get that 
\begin{align}\label{cauchy_schwartz_lebesgue_advanced}
    \frac{1}{\mathcal{L}^{d+1}(\overline{B_\delta(0)})} \Bigg|\int_{\overline{B_\delta(0)}}  \int_{-s}^s |\nabla_g f|_g(y +te_0)|e_0|_g \, \di t \di \mathcal{L}^{d+1}(y) &-\int_{\overline{B_\delta(0)}}\int_{-s}^s   \langle \nabla_g f, e_0\rangle_g(y +te_0) \, \di t \di \mathcal{L}^{d+1}(y)\Bigg|\nonumber \\
    &\leq 2s(4\lambda + 3|e_0|_{g(0)})\epsilon.
\end{align}
Using furthermore \eqref{lebesgue_for_k}, \eqref{lambda_minimal_lor_length_e_0}, and \eqref{lebesgue_advanced_for_k} with $k \in \{\min(\e, |\di f|-G), G\}$ yields
\begin{align}\label{contradiction_assumption_lebesgue_advanced}
    \frac{1}{\mathcal{L}^{d+1}(\overline{B_\delta(0)})} \int_{\overline{B_\delta(0)}}  \int_{-s}^s \big(|\di f|-|\nabla_g f|_g(y +te_0)\big)|e_0|_g \, \di t \di \mathcal{L}^{d+1}(y) \geq 2s(|e_0|_{g(0)}\e - 4\lambda\epsilon).
\end{align}
Recall that our choice of $W$ (see the discussion leading up to \eqref{lambda_minimal_lor_length_e_0}) allows us to apply Lemma \ref{fundamental_thm_of_calculus_for_measure}.
Now, Lemma \ref{identify_derivative_with_absolutely_continuous_part} together with Lemma \ref{fundamental_thm_of_calculus_for_measure} and \eqref{lebesgue_assumption_for_dfs} gives that 
\begin{align}\label{comparison_with_distributional_derivative}
    &\frac{1}{\mathcal{L}^{d+1}(\overline{B_\delta(0)})} \int_{\overline{B_\delta(0)}}  \int_{-s}^s \langle \nabla_g f, e_0\rangle_g(y +te_0) \, \di t \di \mathcal{L}^{d+1}(y) \nonumber \\
    &= \frac{1}{\mathcal{L}^{d+1}(\overline{B_\delta(0)})} \int_{-s}^s \int_{\overline{B_\delta(0)}}  \, \di \mathcal{D}f(e_0)(y +te_0) \di t - \frac{1}{\mathcal{L}^{d+1}(\overline{B_\delta(0)})} \int_{-s}^s \int_{\overline{B_\delta(p)}}  \, \di \mathcal{D}f^s(e_0)(y +te_0) \di t \nonumber \\
    &\geq \frac{1}{\mathcal{L}^{d+1}(\overline{B_\delta(0)})} \int_{\overline{B_\delta(0)}} f(y+se_0)-f(y-se_0) \, \di \mathcal{L}^{d+1}(y) -8 s\lambda \epsilon.  
\end{align}
Then, a similar computation as in \eqref{evaluate_test_plan_at_G} followed by the combination of \eqref{cauchy_schwartz_lebesgue_advanced}, \eqref{contradiction_assumption_lebesgue_advanced}, and \eqref{comparison_with_distributional_derivative}, yields that
\begin{align}
     &\int \int_0^1 |\di f|(\gamma_t)|\dot{\gamma}_t|\, \di t \di \boldsymbol{\pi}(\gamma) \\
     &=\frac{1}{\mathcal{L}^{d+1}(B_\delta(0))} \int_{\overline{B_\delta(0)}}  \int_{-s}^s |\di f|(y+se_0)|e_0|_{g(y+se_0)}\,\di t\di \mathcal{L}^{d+1}(y) \nonumber \\
    &\geq \frac{1}{\mathcal{L}^{d+1}(\overline{B_\delta(0)})} \int_{\overline{B_\delta(0)}} f(y+se_0)-f(y-se_0)\, \di \mathcal{L}^{d+1}(y) + 2s(|e_0|_{g(0)}\e - (8\lambda+3|e_0|_{g(0)})\epsilon) \nonumber\\
    &> \frac{1}{\mathcal{L}^{d+1}(\overline{B_\delta(0)})} \int_{\overline{B_\delta(0)}} f(y+se_0)-f(y-se_0)\, \di \mathcal{L}^{d+1}(y) =\int f(\gamma_1)-f(\gamma_0)\, \di \boldsymbol{\pi}(\gamma),
\end{align}
where the last step follows from \eqref{choice_of_epsilon}. Hence, $|\di f|$ is not a weak subslope. This is a contradiction and finishes the proof.
\end{proof}
This allows us to prove the main theorem:
\smallskip

\begin{theorem}[Theorem \ref{mainthm}]
    Let $M$ be a $(d+1)$-dimensional manifold and $g$ be a continuous Lorentzian metric on $M$ such that $(M, g)$ is a causally simple spacetime.
    Then the metric measure spacetime $(M, \ell_g, \vol_g)$ is infinitesimally Minkowskian. 
\end{theorem}
\begin{proof}
    Take two causal functions $f_1, f_2$ on $(M, \ell_g, \vol_g)$. Note that
    \begin{align*}
        \{|f_1+f_2|<\infty\}= \{|2f_1+f_2|<\infty\} = \{|f_1|<\infty\}\cap \{|f_2|<\infty\}.
    \end{align*}
    We aim to prove
    \begin{align}\label{final_minkowskianity}
        2|\di(f_1+f_2)|^2+ 2|\di f_1|^2 = |\di f_2|^2 + |\di (2f_1+f_2)|^2 \ \vol_g-\mathrm{a.e.}
    \end{align}
     Note furthermore that $f_1+f_2$ and $2f_1+f_2$ are differentiable wherever both $f_1$ and $f_2$ are. 
    Define
    \begin{align}
        N:= (\{|f_1|<\infty\}\cap \{|f_2|<\infty\})\setminus \{f_1 \ \mathrm{and}\ f_2\ \mathrm{are\ differentiable}\}.
    \end{align}
    Using that $N$ is a $\vol_g$-null set, Proposition \ref{weak_subslope_identified} yields that in $\{|f_1|<\infty\}\cap \{|f_2|<\infty\}$ it holds
    \begin{align}
        &2|\di(f_1+f_2)|^2+ 2|\di f_1|^2 = 2|\nabla_g(f_1+f_2)|_g^2+ 2|\nabla_g f_1|_g^2 \nonumber\\
        &= |\nabla_g f_2|_g^2 + |\nabla_g (2f_1+f_2)|_g^2=  |\di f_2|^2 + |\di (2f_1+f_2)|^2 \ \vol_g-\mathrm{a.e.}
    \end{align}
    Indeed, the middle equality follows from the fact that taking derivatives is linear and that $g$ is a bilinear form. In any other case, both sides of \eqref{final_minkowskianity} are infinite and hence equal. This finishes the proof.
\end{proof}

\phantomsection
\addcontentsline{toc}{section}{References}
\bibliography{bibliographie}
\bibliographystyle{abbrv}
\end{document}